\documentclass{amsart}
\usepackage{amssymb}

\usepackage{graphicx}
\usepackage{tikz}
\usepackage{float}
\usepackage{subfig}

\newtheorem{theorem}{Theorem}[subsection]
\theoremstyle{plain} \numberwithin{equation}{subsection}

\newtheorem{corollary}[theorem]{Corollary}

\newtheorem{definition}[theorem]{Definition}

\newtheorem{lemma}[theorem]{Lemma}

\newtheorem{problem}[]{Problem}
\newtheorem{proposition}[theorem]{Proposition}
\newtheorem{remark}[theorem]{Remark}

\input{tcilatex}

\begin{document}
\title[Ramified transportation in metric spaces]{Ramified optimal
transportation in geodesic metric spaces}
\author{Qinglan Xia}
\address{University of California at Davis\\
Department of Mathematics\\
Davis,CA,95616}
\email{qlxia@math.ucdavis.edu}
\urladdr{http://math.ucdavis.edu/\symbol{126}qlxia}
\thanks{This work is supported by an NSF grant DMS-0710714.}
\maketitle

\begin{abstract}
An optimal transport path may be viewed as a geodesic in the space of
probability measures under a suitable family of metrics. This geodesic may
exhibit a tree-shaped branching structure in many applications such as
trees, blood vessels, draining and irrigation systems. Here, we extend the
study of ramified optimal transportation between probability measures from
Euclidean spaces to a geodesic metric space. We investigate the existence as
well as the behavior of optimal transport paths under various properties of
the metric such as completeness, doubling, or curvature upper boundedness.
We also introduce the transport dimension of a probability measure on a
complete geodesic metric space, and show that the transport dimension of a
probability measure is bounded above by the Minkowski dimension and below by
the Hausdorff dimension of the measure. Moreover, we introduce a metric,
called ``the dimensional distance", on the space of probability measures.
This metric gives a geometric meaning to the transport dimension: with
respect to this metric, the transport dimension of a probability measure
equals to the distance from it to any finite atomic probability measure.

\end{abstract}

The optimal transportation problem aims at finding an optimal way to
transport a given measure into another with the same mass. In contrast to
the well-known Monge-Kantorovich problem (e.g. \cite{Ambrosio}, \cite%
{Brenier}, \cite{caffarelli}, \cite{evan2}, \cite{mccann}, \cite{kantorovich}%
, \cite{monge}, \cite{villani}), the ramified optimal transportation problem
aims at modeling a branching transport network by an optimal transport path
between two given probability measures. An essential feature of such a
transport path is to favor transportation in groups via a nonlinear
(typically concave) cost function on mass. Transport networks with branching
structures are observable not only in nature as in trees, blood vessels,
river channel networks, lightning, etc. but also in efficiently designed
transport systems such as used in railway configurations and postage
delivery networks. Several different approaches have been done on the
ramified optimal transportation problem in Euclidean spaces, see for
instance \cite{gilbert}, \cite{xia1}, \cite{msm}, \cite{xia2}, \cite{xia3}, %
\cite{BCM}, \cite{buttazzo}, \cite{xia4}, \cite{Solimini}, \cite{book}, \cite%
{xia5}, and \cite{xia6}. Related works on flat chains may be found in \cite%
{white}, \cite{DH}, \cite{xia2} and \cite{paolini}.

This article aims at extending the study of ramified optimal transportation
from Euclidean spaces to metric spaces. Such generalization is not only
mathematically nature but also may be useful for considering specific
examples of metric spaces later. By exploring various properties of the
metric, we show that many results about ramified optimal transportation is
not limited to Euclidean spaces, but can be extended to metric spaces with
suitable properties on the metric. Some results that we prove in this
article are summarized here:

When $X$ is a geodesic metric space, we define a family of metrics $%
d_{\alpha }$ on the space $\mathcal{A}\left( X\right) $ of atomic
probability measures on $X$ for a (possibly negative) parameter
$\alpha <1$. The space $\left( \mathcal{A}\left( X\right) ,d_{\alpha
}\right) $ is still a geodesic metric space when $0\leq \alpha <1$.
A geodesic, also called an optimal transport path, in this space is
a weighted directed graph whose edges are geodesic segments.

Moreover, when $X$ is a geodesic metric space of curvature bounded above, we
find in \S 2, a universal lower bound depending only on the parameter $%
\alpha $ for each comparison angle between edges of any optimal transport
path. If in addition $X$ is a doubling metric space, we show that the degree
of any vertex of an optimal transport path in $X$ is bounded above by a
constant depending only on $\alpha $ and the doubling constant of $X$. On
the other hand, we also provide a lower bound of the curvature of $X$ by a
quantity related to the degree of vertices.

Furthermore, when $X$ is a complete geodesic metric space, we consider
optimal transportation between any two probability measures on $X$ by
considering the completion of the metric space $\left( \mathcal{A}\left(
X\right) ,d_{\alpha }\right) $. A geodesic, if it exists, in the completed
metric space is viewed as an $\alpha -$optimal transport path between
measures. The existence of an optimal transport path is closely related to
the dimensional information of the measures. As a result, we consider the
dimension of measures on $X$ by introducing a new concept called the \textit{%
transport dimension} of measures, which is analogous to the irrigational
dimension of measures in Euclidean spaces studied by \cite{Solimini}. We
show in \S 4.2.3 and 4.3.4 that the transport dimension of a measure is
bounded below by its Hausdorff dimension and above by its Minkowski
dimension. Furthermore, we show that the transport dimension has an
interesting geometric meaning: under a metric (called the \textit{%
dimensional distance}), the transport dimension of a probability measure
equals to the distance from it to any atomic probability measure.

In \S 5, when $X$ is a compact geodesic doubling metric space with Assouad
dimension $m$ and the parameter $\alpha >\max \left\{ 1-\frac{1}{m}%
,0\right\} $, then we show that the space $\mathcal{P}\left( X\right) $ of
probability measures on $X$ with respect to $d_{\alpha } $ is a geodesic
metric space. In other words, there exists an $\alpha $-optimal transport
path between any two probability measures on $X$.

\section{The $d_{\protect\alpha }$ metrics on atomic probability measures on
a metric space}

\subsection{Transport paths between atomic measures}

We first extend some basic concepts about transport paths between measures
of equal mass as studied in \cite{xia1}, with some necessary modifications,
from Euclidean spaces to a metric space.

Let $\left( X,d\right) $ be a geodesic metric space. Recall that a (finite,
positive) atomic measure on $X$ is in the form of
\begin{equation}
\mathbf{a=}\sum_{i=1}^{k}m_{i}\delta _{x_{i}}  \label{prob_meas_a}
\end{equation}%
with distinct points $x_{i}\in X$ and positive numbers $m_{i},$ where $%
\delta _{x}$ denotes the Dirac mass located at the point $x$. The measure $%
\mathbf{a}$ is a probability measure if the mass $\sum_{i=1}^{k}m_{i}=1$.
Let $\mathcal{A}(X)$ be the space of all atomic probability measures on $X$.

\begin{definition}
Given two atomic measures
\begin{equation}
\mathbf{a}=\sum_{i=1}^{k}m_{i}\delta _{x_{i}}\text{ and }\mathbf{b}%
=\sum_{j=1}^{\ell}n_{j}\delta _{y_{j}}  \label{prob_meas}
\end{equation}%
on $X$ of the same mass, a \textbf{transport path} from $\mathbf{a}$ to $%
\mathbf{b}$ is a weighted directed acyclic graph $G$ consisting of a vertex
set $V(G)$, a directed edge set $E(G)$ and a weight function $%
w:E(G)\rightarrow (0,+\infty )$ such that $\{x_{1},x_{2},...,x_{k}\}\cup
\{y_{1},y_{2},...,y_{n}\}\subset V(G)$ and for any vertex $v\in V(G)$, there
is a balance equation
\begin{equation}
\sum_{e\in E(G),e^{-}=v}w(e)=\sum_{e\in E(G),e^{+}=v}w(e)+\left\{
\begin{array}{c}
m_{i},\text{\ if }v=x_{i}\text{\ for some }i=1,...,k \\
-n_{j},\text{\ if }v=y_{j}\text{\ for some }j=1,...,n \\
0,\text{\ otherwise }%
\end{array}%
\right.  \label{path}
\end{equation}%
where each edge $e\in E\left( G\right) $ is a geodesic segment in $X$ from
the starting endpoint $e^{-}$ to the ending endpoint $e^{+}$.
\end{definition}

Note that the balance equation (\ref{path}) simply means the conservation of
mass at each vertex. In terms of polyhedral chains, we simply have $\partial
G=b-a$.

Here, a directed graph $G$ is called \textit{acyclic} if it contains no
directed cycles in the sense that for any vertex $v\in V\left( G\right) $,
there does not exist a list of vertices $\left\{ v_{1},v_{2,}\cdots
,v_{k}\right\} $ such that $v_{1}=v_{k}=v$ and $\left[ v_{i},v_{i+1}\right] $
is a directed edge in $E\left( G\right) $ for each $i=1,2,\cdots ,k-1$.
Reasons for introducing this constraint were given in \cite[Remark 2.1.5]%
{xia6}.

For any two atomic measures $\mathbf{a}$ and $\mathbf{b}$ on $X$ of equal
mass, let \textit{Path}$(\mathbf{a},\mathbf{b})$ be the space of all
transport paths from $\mathbf{a}$ to $\mathbf{b}$. Now, we define the
transport cost for each transport path as follows.

\begin{definition}
For any real number $\alpha \in (-\infty,1]$ and any transport path $G\in \text{Path}(%
\mathbf{a},\mathbf{b})$, we define
\begin{equation*}
\mathbf{M}_{\alpha }(G):=\sum_{e\in E(G)}w(e)^{\alpha }\text{length}(e).
\end{equation*}
\end{definition}

We now consider the following optimal transport problem:

\begin{problem}
Given two atomic measures $\mathbf{a}$ and $\mathbf{b}$ of equal mass on a
geodesic metric space $X$, find a minimizer of
\begin{equation*}
\mathbf{M}_{\alpha }(G)
\end{equation*}%
among all transport paths $G\in Path\left( \mathbf{a},\mathbf{b}\right)$.
\end{problem}

An $\mathbf{M}_{\alpha }$ minimizer in $Path(\mathbf{a},\mathbf{b})$ is
called an $\alpha -$optimal transport path from $\mathbf{a}$ to $\mathbf{b}$.

\subsection{The $d_{\protect\alpha }$ metrics}

\begin{definition}
\label{d_alpha}For any $\alpha \in (-\infty,1]$, we define
\begin{equation*}
d_{\alpha }\left( \mathbf{a},\mathbf{b}\right) =\inf \left\{ \mathbf{M}%
_{\alpha }\left( G\right) :G\in Path\left( \mathbf{a},\mathbf{b}\right)
\right\}
\end{equation*}%
for any $\mathbf{a},\mathbf{b\in }\mathcal{A}(X)$.
\end{definition}

\begin{remark}
Let $\mathbf{\bar{a}}$ and $\mathbf{\bar{b}}$ be two atomic measures
of equal mass $\Lambda>0 $, and let $\mathbf{a}=\frac{1}{\Lambda
}\mathbf{\bar{a}} $ and $\mathbf{b}=\frac{1}{\Lambda
}\mathbf{\bar{b}}$ be the normalization
of $\mathbf{\bar{a}}$ and $\mathbf{\bar{b}}$. Then, for any transport path $%
\bar{G}\in Path\left( \mathbf{\bar{a}},\mathbf{\bar{b}}\right) $, we have $%
G=\left\{ V\left( \bar{G}\right) ,E\left( \bar{G}\right) ,\frac{1}{\Lambda }%
w\right\} $ is a transport path from $\mathbf{a}$ to $\mathbf{b}$ with $%
\mathbf{M}_{\alpha }(\bar{G})=\Lambda ^{\alpha }\mathbf{M}_{\alpha }(G)$.
Thus, we also set $d_{\alpha }\left( \mathbf{\bar{a}},\mathbf{\bar{b}}%
\right) =\Lambda ^{\alpha }d_{\alpha }\left( \mathbf{a},\mathbf{b}\right) .$
\end{remark}

It is easy to see that $d_{\alpha }$ is a metric on $\mathcal{A}(X)$ when $%
0\leq \alpha \leq 1$. But to show that $d_{\alpha }$ is still a metric when $%
\alpha <0$, we need some estimates on the lower bound of $d_{\alpha }\left(
\mathbf{a},\mathbf{b}\right) $ when $\mathbf{a}\neq \mathbf{b}$.

We denote $S\left( p,r\right) $ (and $\bar{B}\left( p,r\right)$,
respectively) to be the sphere (and the closed ball, respectively) centered
at $p\in X$ of radius $r>0$. Note that for any transport path $G$, the
restriction of $G$ on any closed ball $\bar{B}\left( p,r_{0}\right) $ gives
a transport path $G|_{\bar{B}\left( p,r_{0}\right) }$ between the
restriction of measures.

\begin{lemma}
\label{estimate}Suppose $\mathbf{a}$ and $\mathbf{b}$ are two atomic
measures on a geodesic metric space $X$ of equal total mass, and $G$ is a
transport path from $\mathbf{a}$ to $\mathbf{b}$. For each $p\in X$, if the
intersection of $G\cap S\left( p,r\right) $ as sets is nonempty for almost
all $r\in \left[ 0,r_{0}\right] $ for some $r_{0}>0$, then
\begin{equation}
\mathbf{M}_{\alpha }(G|_{\bar{B}\left( p,r_{0}\right) })\geq
\int_{0}^{r_{0}}\sum_{e\in \mathbf{E}_{r}}\left[ w(e)\right] ^{\alpha }dr,
\label{key_estimate}
\end{equation}%
where for each $r$, the set%
\begin{equation*}
\mathbf{E}_{r}:=\left\{ e\in E\left( G\right) :e\cap S\left( p,r\right) \neq
\emptyset \right\}
\end{equation*}%
is the family of all edges of $G$ that intersects with the sphere $S\left(
p,r\right) $.
\end{lemma}

\begin{proof}
For every edge $e$ of $G$, let $p^{\ast }$ and $p_{\ast }$ be the points on $%
e$ such that
\begin{equation*}
d\left( p,p^{\ast }\right) =\max \left\{ d\left( p,x\right) :x\in e\right\}
\text{ and }d\left( p,p_{\ast }\right) =\min \left\{ d\left( p,x\right)
:x\in e\right\} .
\end{equation*}%
Then, since $e$ is a geodesic segment in $X$,
\begin{equation*}
length\left( e\right) \geq d\left( p^{\ast },p^{\ast }\right) \geq \left|
d\left( p,p^{\ast }\right) -d\left( p,p_{\ast }\right) \right|
=\int_{0}^{\infty }\chi _{I_{e}}\left( r\right) dr
\end{equation*}%
where $\chi _{I_{e}}\left( r\right) $ is the characteristic function of the
interval $I_{e}:=\left[ d\left( p,p_{\ast }\right) ,d\left( p,p^{\ast
}\right) \right] $. By assumption, $\mathbf{E}_{r}$ is nonempty for almost
all $r\in \left[ 0,r_{0}\right] $. Also, observe that $e\in \mathbf{E}_{r}$
if and only if $\chi _{I_{e}}\left( r\right) =1$. Therefore,
\begin{eqnarray*}
\mathbf{M}_{\alpha }(G|_{\bar{B}\left( p,r_{0}\right) }) &=&\sum_{e\in
E\left( G\mid B\left( p,r_{0}\right) \right) }\left[ w(e)\right] ^{\alpha }%
\text{length}(e) \\
&\geq &\sum_{e\in E\left( G\mid B\left( p,r_{0}\right) \right) }\left[ w(e)%
\right] ^{\alpha }\int_{0}^{\infty }\chi _{I_{e}}\left( r\right) dr \\
&=&\int_{0}^{r_{0}}\sum_{e\in \mathbf{E}_{r}}\left[ w(e)\right] ^{\alpha }dr.
\end{eqnarray*}
\end{proof}

The following corollary implies a positive lower bound on $d_{\alpha }\left(
\mathbf{a,b}\right) $ when $\mathbf{a\neq b}$.

\begin{corollary}
\label{negative_estimate}Let the assumptions be as in Lemma \ref{estimate}
and $\alpha \leq 0$. Then
\begin{equation*}
\mathbf{M}_{\alpha }\left( G|_{\bar{B}\left( p,r_{0}\right) }\right) \geq
\Lambda ^{\alpha }r_{0},
\end{equation*}%
where $\Lambda $ is an upper bound of the weight $w\left( e\right) $ for
every edge $e$ in $G|_{\bar{B}\left( p,r_{0}\right) }$. In particular, for
any atomic measure
\begin{equation*}
\mathbf{a=}\sum_{i=1}^{k}m_{i}\delta _{x_{i}}
\end{equation*}%
on $X$ with mass $\left| \left| \mathbf{a}\right| \right|
:=\sum_{i=1}^{k}m_{i}>0$, we have the following estimate
\begin{equation}
d_{\alpha }\left( \mathbf{a,\left| \left| \mathbf{a}\right| \right| \delta }%
_{p}\right) \geq \left| \left| \mathbf{a}\right| \right| ^{\alpha
}\max_{1\leq i\leq k}\left\{ d\left( p,x_{i}\right) \right\} \text{ .}
\label{lower_bound}
\end{equation}
\end{corollary}

\begin{proof}
When $\alpha \leq 0$, we have \ for each $r$ with $\mathbf{E}_{r}$ nonempty,%
\begin{equation*}
\sum_{e\in \mathbf{E}_{r}}\left[ w(e)\right] ^{\alpha }\geq \max_{e\in
\mathbf{E}_{r}}\left[ w(e)\right] ^{\alpha }\geq \Lambda ^{\alpha }
\end{equation*}%
where $\Lambda $ is any upper bound of the weights of edges in $G|_{\bar{B}%
\left( p,r_{0}\right) }$. Therefore, by (\ref{key_estimate}),
\begin{equation*}
\mathbf{M}_{\alpha }\left( G|_{\bar{B}\left( p,r_{0}\right) }\right) \geq
\Lambda ^{\alpha }r_{0}\text{.}
\end{equation*}%
Now, let $G\in Path\left( \mathbf{a,\left| \left| \mathbf{a}\right| \right|
\delta }_{p}\right) $. When $\alpha \leq 0$, by the acyclic property of
transport paths, we have
\begin{equation*}
w\left( e\right) \leq \left| \left| \mathbf{a}\right| \right|
\end{equation*}%
for every edge $e\in E\left( G\right) $. Then, (\ref{lower_bound}) follows
by setting $r_{0}=\max_{1\leq i\leq k}\left\{ d\left( p,x_{i}\right)
\right\} $ and $\Lambda =$ $\left| \left| \mathbf{a}\right| \right| $.
\end{proof}

Lemma \ref{estimate} also gives a lower bound estimate for positive $\alpha $%
, which will be used in proposition \ref{general_positive_estimate}.

\begin{corollary}
\label{positive_estimate}Suppose $0\leq \alpha <1$. For any $\mathbf{a\in }%
\mathcal{A}\left( X\right) $ in the form of (\ref{prob_meas_a}), $p\in X$
and $r_{0}>0$, we have
\begin{equation}
\left[ \sum_{d\left( p,x_{i}\right) >r_{0}}m_{i}\right] ^{\alpha }\leq \frac{%
d_{\alpha }(\mathbf{a},\delta _{p})}{r_{0}}.  \label{positive_estimate_eqn}
\end{equation}
\end{corollary}

\begin{proof}
Let $\lambda =\sum_{d\left( p,x_{i}\right) >r_{0}}m_{i}$. Let $G$ be any
transport path from $\mathbf{a}$ to $\delta _{p}$. Then, for any $0<r\leq
r_{0}$, we have
\begin{equation*}
\sum_{e\in \mathbf{E}_{r}}w(e)\geq \lambda .
\end{equation*}%
By lemma \ref{estimate}, since the function $f\left( x\right) =x^{\alpha }$
is concave on $\left[ 0,1\right] $ when $0\leq \alpha <1$, we have
\begin{eqnarray*}
\mathbf{M}_{\alpha }(G|_{\bar{B}\left( p,r_{0}\right) }) &\geq
&\int_{0}^{r_{0}}\sum_{e\in \mathbf{E}_{r}}\left[ w(e)\right] ^{\alpha }dr \\
&\geq &\int_{0}^{r_{0}}\left[ \sum_{e\in \mathbf{E}_{r}}w(e)\right] ^{\alpha
}dr\geq \int_{0}^{r_{0}}\lambda ^{\alpha }dr=\lambda ^{\alpha }r_{0}\text{.}
\end{eqnarray*}%
Therefore, we have (\ref{positive_estimate_eqn}).
\end{proof}

By means of corollary \ref{negative_estimate}, the proof of %
\cite[Proposition 2.2.3]{xia6} shows the following proposition.

\begin{proposition}
Suppose $\alpha <1$, and $X$ is a geodesic metric space. Then $d_{\alpha }$
defined in definition \ref{d_alpha} is a metric on the space $\mathcal{A}(X)$
of atomic probability measures on $X$.
\end{proposition}

\subsection{The $d_{\protect\alpha }$ metric viewed as a metric induced by a
quasimetric}

When $0\leq \alpha <1$, another approach of the metric $d_{\alpha }$ was
introduced in \cite{xia5}, which says that the metric $d_{\alpha }$ is the
intrinsic metric on $\mathcal{A}\left( X\right) $ induced by a quasimetric
\footnote{%
A function $q:X\times X\rightarrow \lbrack 0,+\infty )$ is a \textit{%
quasimetric} on $X$ if $q$ satisfies all the conditions of a metric except
that $q$ satisfies a relaxed triangle inequality $q\left( x,y\right) \leq
C\left( q\left( x,z\right) +q\left( z,y\right) \right) $ for some $C\geq 1$,
rather than the usual triangle inequality.} $J_{\alpha }$. Let us briefly
recall the definition of the quasimetric $J_{\alpha }$ here.

Let $\mathbf{a}$ and $\mathbf{b}$ be two fixed atomic probability measures
in the form of (\ref{prob_meas}) on a metric space $X$, a \textit{transport
plan} from $\mathbf{a}$ to $\mathbf{b}$ is an atomic probability measure
\begin{equation}
\gamma =\sum_{i=1}^{m}\sum_{j=1}^{\ell }\gamma _{ij}\delta _{\left(
x_{i},y_{j}\right) }  \label{transport_plan}
\end{equation}%
in the product space $X\times X$ such that
\begin{equation}
\sum_{i=1}^{m}\gamma _{ij}=n_{j}\text{ and }\sum_{j=1}^{\ell }\gamma
_{ij}=m_{i}  \label{margins}
\end{equation}%
for each $i$ and $j$. Let $Plan\left( \mathbf{a},\mathbf{b}\right) $ be the
space of all transport plans from $\mathbf{a}$ to $\mathbf{b}$.

For any atomic probability measure $\gamma $ in $X\times X$ of the form (\ref%
{transport_plan}) and any $0\leq \alpha <1$, we define
\begin{equation*}
H_{\alpha }\left( \gamma \right) :=\sum_{i=1}^{m}\sum_{j=1}^{\ell }\left(
\gamma _{ij}\right) ^{\alpha }d\left( x_{i},y_{j}\right) ,
\end{equation*}%
where $d$ is the given metric on $X$.

Using $H_{\alpha }$, we define
\begin{equation*}
J_{\alpha }\left( \mathbf{a},\mathbf{b}\right) :=\min \left\{ H_{\alpha
}\left( \gamma \right) :\gamma \in Plan\left( \mathbf{a},\mathbf{b}\right)
\right\} .
\end{equation*}%
For any given natural number $N\in \mathbb{N}$ , let $\mathcal{A}_{N}(X)$ be
the space of all atomic probability measures
\begin{equation*}
\sum_{i=1}^{m}a_{i}\delta _{x_{i}}
\end{equation*}%
on $X$ with $m\leq N$, and then $\mathcal{A}\left( X\right) =\bigcup_{N}%
\mathcal{A}_{N}\left( X\right) $ is the space of all atomic probability
measures on $X$.

In \cite[Proposition 4.2]{xia5}, we showed that $J_{\alpha }$ defines a
quasimetric on $\mathcal{A}_{N}\left( X\right) $. Moreover, $J_{\alpha }$ is
a complete quasimetric on $\mathcal{A}_{N}\left( X\right) $ if $\left(
X,d\right) $ is a complete metric space. The quasimetric $J_{\alpha }$ has a
very nice property in the sense that this quasimetric is able to induce an
intrinsic metric on $\mathcal{A}_{N}\left( X\right) $.

\begin{proposition}
\cite[Theorem 4.17, Corollary 4.18]{xia5}Suppose $0\leq \alpha <1$, and $X$
is a geodesic metric space. Then, the metric $d_{\alpha }$ defined in
definition (\ref{d_alpha}) is the intrinsic metric on $\mathcal{A}_{N}\left(
X\right) $ induced by the quasimetric $J_{\alpha }$.
\end{proposition}

Moreover, in \cite[remark 4.16]{xia5} we have a simple formula for the $%
\mathbf{M}_{\alpha }$ cost. Suppose $G\in Path\left( \mathbf{a,b}\right) $
for some $\mathbf{a,b\in }\mathcal{A}_{N}\left( X\right) $. If each edge of $%
G$ is a geodesic curve between its endpoints in the geodesic metric space $X$%
, then there exists an associated piecewise metric Lipschitz curve $g:\left[
0,1\right] \rightarrow $ $\mathcal{A}_{N}\left( X\right) $ such that
\begin{equation*}
\mathbf{M}_{\alpha }\left( G\right) =\int_{0}^{1}\left| \dot{g}\left(
t\right) \right| _{J_{\alpha }}dt.
\end{equation*}%
where the quasimetric derivative
\begin{equation*}
\left| \dot{g}\left( t\right) \right| _{J_{\alpha }}:=\lim_{s\rightarrow t}%
\frac{J_{\alpha }\left( g\left( t\right) ,g\left( s\right) \right) }{\left|
t-s\right| }
\end{equation*}%
exists almost everywhere.

\begin{corollary}
\label{length_space_proof} Suppose $\left( X,d\right) $ is a complete
geodesic metric space. Then, $(\mathcal{A}_{N}\left( X\right) ,d_{\alpha })$
is a complete geodesic metric space for each $0\leq \alpha <1$.
\end{corollary}

Since $\mathcal{A}_{1}\left( X\right) \subset \mathcal{A}_{2}\left( X\right)
\subset \cdots \subset \mathcal{A}_{N}\left( X\right) \subset \cdots $, and $%
(\mathcal{A}_{N}\left( X\right) ,d_{\alpha })$ is a geodesic space for each $%
N$, we have the following existence result of optimal transport path:

\begin{proposition}
\label{all_atomic_measures} \cite[proposition 4.02]{xia5} Suppose $\left(
X,d\right) $ is a complete geodesic metric space. Then, $\left( \mathcal{A}%
\left( X\right) ,d_{\alpha }\right) $ is a geodesic metric space for each $%
0\leq \alpha <1$. Moreover, for any $\mathbf{a},\mathbf{b\in }\mathcal{A}%
\left( X\right) $, every $\alpha -$optimal transport path from $\mathbf{a}$
to $\mathbf{b}$ is a geodesic from $\mathbf{a}$ to $\mathbf{b}$ in the
geodesic space $\left( \mathcal{A}\left( X\right) ,d_{\alpha }\right) $.
Vice versa, every geodesic from $\mathbf{a}$ to $\mathbf{b}$ in $\left(
\mathcal{A}\left( X\right) ,d_{\alpha }\right) $ is an $\alpha -$optimal
transport path from $\mathbf{a}$ to $\mathbf{b}$.
\end{proposition}

\section{Transportation in metric spaces with curvature bounded above}

In this section, we will show that when $X$ is a geodesic metric space with
curvature bounded above, then there exists a universal upper bound for the
degree of every vertex of every optimal transport path on $X$.

We now recall the definition of a space of bounded curvature \cite%
{metricgeometry}. For a real number $k$, the model space $M_{k}^{2}$ is the
simply connected surface with constant curvature $k$. That is, if $k=0$,
then $M_{k}^{2}$ is the Euclidean plane. If $k>0$, then $M_{k}^{2}$ is
obtained from the sphere $\mathbb{S}^{2}$ by multiplying the distance
function by the constant $\frac{1}{\sqrt{k}}$. If $k<0$, then $M_{k}^{2}$ is
obtained from the hyperbolic space $\mathbb{H}^{2}$ by multiplying the
distance function by the constant $\frac{1}{\sqrt{-k}}$. The diameter of $%
M_{k}^{2}$ is denoted by $D_{k}:=\pi /\sqrt{k}$ for $k>0$ and $D_{k}:=\infty
$ for $k\leq 0$.

Let $\left( X,d\right) $ be a geodesic metric space, and let $\Delta ABC$ be
a geodesic triangle in $X$ with geodesic segments as its sides. A \textit{%
comparison} triangle $\Delta \bar{A}\bar{B}\bar{C}$ is a triangle in the
model space $M_{k}^{2}$ such that $d\left( A,B\right) =\left| \bar{A}-\bar{B}%
\right| _{k},$ $d\left( B,C\right) =\left| \bar{B}-\bar{C}\right| _{k}$ and $%
d\left( A,C\right) =\left| \bar{A}-\bar{C}\right| _{k}$, where $\left| \cdot
\right| _{k}$ denotes the distance function in the model space $M_{k}^{2}$.
Such a triangle is unique up to isometry. Also, the interior angle of $%
\Delta \bar{A}\bar{B}\bar{C}$ at $\bar{B}$ is called the \textit{comparison
angle} between $A$ and $C$ at $B$.

A geodesic metric space $\left( X,d\right) $ is \textit{a space of curvature
bounded above by a real number }$k$ if for every geodesic triangle $\Delta
ABC$ in $X$ and every point $h$ in the geodesic segment $\gamma _{AC}$, one
has
\begin{equation*}
d\left( h,B\right) \leq \left| \bar{h}-\bar{B}\right| _{k}
\end{equation*}%
where $\bar{h}$ is the point on the side $\gamma _{\bar{A}\bar{C}}$ of a
comparison triangle $\Delta \bar{A}\bar{B}\bar{C}$ in $M_{k}^{2}$ such that $%
\left| \bar{h}-\bar{C}\right| _{k}=d\left( h,C\right) $.

Now, let $X$ be a geodesic metric space with curvature bounded above by a
real number $k$. Suppose $\alpha <1$ and $G$ is an $\alpha -$optimal
transport path between two atomic probability measures $\mathbf{a,b\in }%
\mathcal{A}\left( X\right) $. We will show that the comparison angle of any
two edges from a common vertex of $G$ is bounded below by a universal
constant depending only on $\alpha $. Moveover, when $X$ is in addition a
doubling space, then the degree of any vertex $v$ of $G$ is bounded above by
a constant depending only on $\alpha $ and the doubling constant of $X$.

\begin{minipage}[h]{0.5\linewidth}
More precisely, let $O$ be any vertex of $G$ and $e_{i}$ be any two distinct
directed edges with $e_{i}^{+}=O$ (or $e_{i}^{-}=O$ simultaneously) and
weight $m_{i}>0$ for $i=1,2$. Also, for $i=1,2$, let $A_{i}$ be the point on
the edge $e_{i}$ with $d\left( O,A_{i}\right) =r$ for some $r$ satisfying $%
0<r\leq \frac{1}{2}D_{k}$ and $r\leq length\left( e_{i}\right)$.
\end{minipage}
\begin{minipage}[h]{0.5\linewidth}
\centering \includegraphics{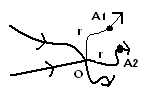}
\label{y_shape}
\end{minipage}
%
%

Now, we want to estimate the distance $d\left( A_{1},A_{2}\right) $. To do
it, we first denote
\begin{equation}
R=\sqrt{\frac{\left( m_{1}^{\alpha }+m_{2}^{\alpha }\right) ^{2}-\left(
m_{1}+m_{2}\right) ^{2\alpha }}{m_{1}^{\alpha }m_{2}^{\alpha }}}
\label{R_def}
\end{equation}%
and have the following estimates for $R$:

\begin{lemma}
For each $\alpha <1$, the infimum of $R$ is given by%
\begin{equation}
R_{\alpha }:=\left\{
\begin{array}{cc}
\sqrt{2}, & \text{if }0<\alpha <\frac{1}{2} \\
\sqrt{4-4^{\alpha }}, & \text{if }\frac{1}{2}\leq \alpha <1\text{ or }\alpha
\leq 0.%
\end{array}%
\right.  \label{R_a}
\end{equation}%
For each $0\leq \alpha <1$, the supremum of $R$ is given by%
\begin{equation*}
\bar{R}_{\alpha }:=\left\{
\begin{array}{cc}
\sqrt{2}, & \text{if }\frac{1}{2}\leq \alpha <1\text{ } \\
\sqrt{4-4^{\alpha }}, & \text{if }0\leq \alpha <\frac{1}{2}.%
\end{array}%
\right.
\end{equation*}%
Also, when $\alpha =0$, then $R\equiv \sqrt{3}$. When $\alpha =\frac{1}{2}$,
then $R\equiv \sqrt{2}$. When $\alpha =1$, then $R\equiv 0$.
\end{lemma}

When $\alpha <0$, we will show $R\leq 2$ later in lemma \ref{R_less_than_2}.

\begin{proof}
We first denote
\begin{equation}
k_{1}=\frac{m_{1}}{m_{1}+m_{2}}\text{, }k_{2}=\frac{m_{2}}{m_{1}+m_{2}}
\label{k_i}
\end{equation}%
as in \cite[Example 2.1]{xia1}. Note that $k_{1}+k_{2}=1$ and
\begin{equation*}
R=\sqrt{\frac{\left( k_{1}^{\alpha }+k_{2}^{\alpha }\right) ^{2}-1}{%
k_{1}^{\alpha }k_{2}^{\alpha }}}.
\end{equation*}%
By considering the function

\begin{equation*}
f_{\alpha }\left( x\right) =\frac{\left( x^{\alpha }+\left( 1-x\right)
^{\alpha }\right) ^{2}-1}{x^{\alpha }\left( 1-x\right) ^{\alpha }}
\end{equation*}%
for $x\in \left( 0,1\right) $ and $\alpha \leq 1$, we have $R=\sqrt{%
f_{\alpha }\left( k_{1}\right) }$. Using Calculus, one may check that for
each $x\in \left( 0,1\right) $,

\begin{enumerate}
\item when $\alpha \in \left( 0,\frac{1}{2}\right) $, the function $%
f_{\alpha }$ is strictly concave up and
\begin{equation*}
4-4^{\alpha }=f_{\alpha }\left( \frac{1}{2}\right) \leq f_{\alpha }\left(
x\right) <\lim_{y\rightarrow 0+}f_{\alpha }\left( y\right) =2;
\end{equation*}

\item when $\alpha \in \left( \frac{1}{2},1\right) $, the function $%
f_{\alpha }$ is strictly concave down and
\begin{equation*}
4-4^{\alpha }=f_{\alpha }\left( \frac{1}{2}\right) \geq f_{\alpha }\left(
x\right) >\lim_{y\rightarrow 0+}f_{\alpha }\left( y\right) =2;
\end{equation*}

\item when $\alpha \in \left( -\infty ,0\right) $, the function $f_{\alpha }$
is strictly concave down and
\begin{equation*}
f_{\alpha }\left( x\right) \geq f_{\alpha }\left( \frac{1}{2}\right)
=4-4^{\alpha };
\end{equation*}

\item $f_{\alpha }$ has constant values when $\alpha \in \left\{ 0,\frac{1}{2%
},1\right\} $.
\end{enumerate}

Using these facts, we get the estimates for $R=\sqrt{f_{\alpha }\left(
k_{1}\right) }$ for each $\alpha $.
\end{proof}

Now, we have the following key estimates for the distance $d\left(
A_{1},A_{2}\right) $:

\begin{lemma}
\label{key_lemma}Assume that $d\left( O,A_{1}\right) =d\left( O,A_{2}\right)
=r$ and $0<r\leq \frac{1}{2}D_{k}$. Then, we have the following estimates
for $a:=d\left( A_{1},A_{2}\right) $:

\begin{enumerate}
\item If $k>0$, then
\begin{equation*}
\cos \left( a\sqrt{k}\right) \leq 1-\frac{R^{2}}{2}\sin ^{2}\left( r\sqrt{k}%
\right) \text{ \ \ i.e. }\sin \frac{a\sqrt{k}}{2}\geq \frac{R}{2}\sin \left(
r\sqrt{k}\right).
\end{equation*}

\item If $k=0$, then
\begin{equation*}
a\geq Rr.
\end{equation*}

\item If $k<0$, then
\begin{equation*}
\cosh \left( a\sqrt{-k}\right) \geq 1+\frac{R^{2}}{2}\sinh ^{2}\left( r\sqrt{%
-k}\right) \text{ i.e. }\sinh \frac{a\sqrt{-k}}{2}\geq \frac{R}{2}\sinh
\left( r\sqrt{-k}\right) .
\end{equation*}
\end{enumerate}
\end{lemma}

\begin{proof}
Let $P$ be the point on the geodesic $\gamma _{A_{1}A_{2}}$ from $A_{1}$ to $%
A_{2}$ with
\begin{equation*}
d\left( A_{i},P\right) =\lambda _{i}d\left( A_{1},B\right)
\end{equation*}%
for $i=1,2$ and some $\lambda _{i}\in \left( 0,1\right) $ to be chosen later
in (\ref{lambda}) with $\lambda _{1}+\lambda _{2}=1$. For any $t\in \left[
0,1\right] $, let $Q\left( t\right) $ be the point on the geodesic $\gamma
_{OP}$ from $O$ to $P$ such that
\begin{equation*}
d\left( O,Q\left( t\right) \right) =tb\text{ and }d\left( P,Q\left( t\right)
\right) =\left( 1-t\right) b
\end{equation*}%
where $b=d\left( O,P\right) $. For $i=1,2$, let $\Delta \bar{A}_{i}\bar{P}%
\bar{O}$ be a comparison triangle of $\Delta A_{i}PO$ in the model space $%
M_{k}^{2}$.
\begin{figure}[h]
\centering \subfloat[A triangle $\Delta A_{i}PO$ in
$X$]{\label{old_triangle}\includegraphics[height=1.75in]{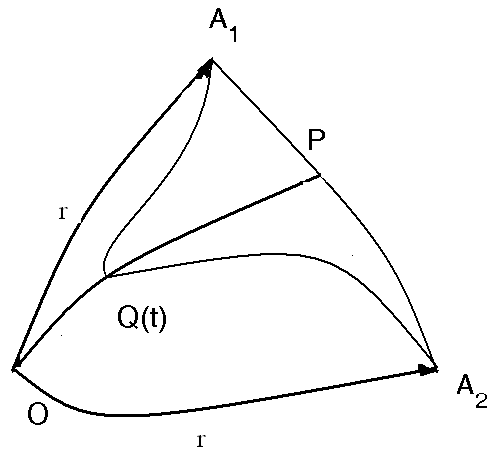}}
\subfloat[A comparison triangle $\Delta \bar{A}_{i}\bar{P}\bar{O}$
in
$M_k^2$]{\label{triangle}\includegraphics[height=1.75in]{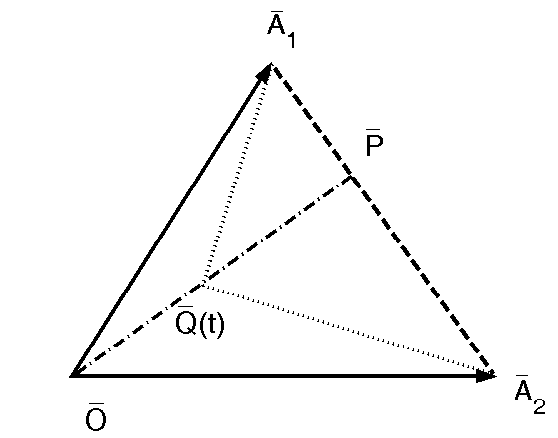}} %
\label{comparison_triangles}
\caption{Comparison triangles}
\end{figure}

Thus,
\begin{equation*}
\left| \bar{A}_{i}-\bar{P}\right| _{k}=\lambda _{i}a,\left| \bar{A}_{i}-%
\bar{O}\right| _{k}=r\text{ and }\left| \bar{O}-\bar{P}\right| _{k}=b.
\end{equation*}%
Let $\bar{Q}_{i}\left( t\right) $ be the point on the side $\gamma _{\bar{O}%
\bar{P}}$ of $\Delta \bar{A}_{i}\bar{P}\bar{O}$ such that $\left| \bar{O}-%
\bar{Q}_{i}\left( t\right) \right| _{k}=d\left( O,Q\left( t\right) \right)
=tb$ and let $\sigma _{i}\left( t\right) =\left| \bar{A}_{i}-\bar{Q}%
_{i}\left( t\right) \right| _{k}$. Since $X$ has curvature bounded above by $%
k$, we have $\sigma _{i}\left( t\right) \geq d\left( A_{i},Q\left( t\right)
\right)$. Let
\begin{eqnarray*}
H\left( t\right) &:&=\frac{1}{\left( m_{1}+m_{2}\right) ^{\alpha }}\left[
m_{1}^{\alpha }\sigma _{1}\left( t\right) +m_{2}^{\alpha }\sigma _{2}\left(
t\right) +\left( m_{1}+m_{2}\right) ^{\alpha }tb\right] \\
&=&k_{1}^{\alpha }\sigma _{1}\left( t\right) +k_{2}^{\alpha }\sigma
_{2}\left( t\right) +tb \\
&\geq &\frac{1}{\left( m_{1}+m_{2}\right) ^{\alpha }}\left[ m_{1}^{\alpha
}d\left( A_{1},Q\left( t\right) \right) +m_{2}^{\alpha }d\left(
A_{2},Q\left( t\right) \right) +\left( m_{1}+m_{2}\right) ^{\alpha }d\left(
O,Q\left( t\right) \right) \right] \\
&\geq &\frac{m_{1}^{\alpha }d\left( O,A_{1}\right) +m_{2}^{\alpha }d\left(
O,A_{2}\right) }{\left( m_{1}+m_{2}\right) ^{\alpha }}=H\left( 0\right) ,
\end{eqnarray*}%
since $O$ is a vertex of an $\alpha -$optimal transport path $G$. This
implies that\thinspace $H^{\prime }\left( 0\right) \geq 0$ if $H^{\prime
}\left( 0\right) $ exists. Now, we may calculate the derivative $H^{\prime
}\left( 0\right) =k_{1}^{\alpha }\sigma _{1}^{\prime }\left( 0\right)
+k_{2}^{\alpha }\sigma _{2}^{\prime }\left( 0\right) +b$ as follows.

When $k>0$, by applying the spherical law of cosines to triangles $\Delta
\bar{A}_{i}\bar{P}\bar{O}$ and $\Delta \bar{A}_{i}\bar{Q}\left( t\right)
\bar{O},$ we have
\begin{eqnarray*}
\cos \left( \lambda _{i}a\sqrt{k}\right) &=&\cos \left( r\sqrt{k}\right)
\cos \left( b\sqrt{k}\right) +\sin \left( r\sqrt{k}\right) \sin \left( b%
\sqrt{k}\right) \cos \theta _{i} \\
\cos \left( \sigma _{i}\left( t\right) \sqrt{k}\right) &=&\cos \left( r\sqrt{%
k}\right) \cos \left( tb\sqrt{k}\right) +\sin \left( r\sqrt{k}\right) \sin
\left( tb\sqrt{k}\right) \cos \theta _{i},
\end{eqnarray*}%
where $\theta _{i}$ is the angle $\measuredangle \bar{A}_{i}\bar{O}\bar{P}$.
Thus,
\begin{equation*}
\sin \left( tb\sqrt{k}\right) \cos \left( \lambda _{i}a\sqrt{k}\right) -\sin
\left( b\sqrt{k}\right) \cos \left( \sigma _{i}\left( t\right) \sqrt{k}%
\right) =-\cos \left( r\sqrt{k}\right) \sin \left( \left( 1-t\right) b\sqrt{k%
}\right) .
\end{equation*}%
Taking derivative with respect to $t$ at $t=0$ and using the fact $\sigma
_{i}\left( 0\right) =r$, we have%
\begin{equation*}
\left( b\sqrt{k}\right) \cos \left( \lambda _{i}a\sqrt{k}\right) +\sin
\left( b\sqrt{k}\right) \sin \left( r\sqrt{k}\right) \sigma _{i}^{\prime
}\left( 0\right) \sqrt{k}=b\sqrt{k}\cos \left( r\sqrt{k}\right) \cos \left( b%
\sqrt{k}\right) .
\end{equation*}%
Therefore, for $i=1,2$,
\begin{equation*}
\sigma _{i}^{\prime }\left( 0\right) =\frac{b\cos \left( r\sqrt{k}\right)
\cos \left( b\sqrt{k}\right) -b\cos \left( \lambda _{i}a\sqrt{k}\right) }{%
\sin \left( b\sqrt{k}\right) \sin \left( r\sqrt{k}\right) }.
\end{equation*}

Applying these expressions to $H^{\prime }\left( 0\right) =k_{1}^{\alpha
}\sigma _{1}^{\prime }\left( 0\right) +k_{2}^{\alpha }\sigma _{2}^{\prime
}\left( 0\right) +b\geq 0$, we have%
\begin{equation}
\left( k_{1}^{\alpha }+k_{2}^{\alpha }\right) \cos \left( r\sqrt{k}\right)
\cos \left( b\sqrt{k}\right) +\sin \left( r\sqrt{k}\right) \sin \left( b%
\sqrt{k}\right) \geq k_{1}^{\alpha }\cos \left( \lambda _{1}a\sqrt{k}\right)
+k_{2}^{\alpha }\cos \left( \lambda _{2}a\sqrt{k}\right) .  \label{V_W}
\end{equation}%
By setting
\begin{equation*}
Ve^{i\Theta _{1}}=\left( k_{1}^{\alpha }+k_{2}^{\alpha }\right) \cos \left( r%
\sqrt{k}\right) +i\sin \left( r\sqrt{k}\right)
\end{equation*}%
as a complex number, we have
\begin{equation*}
\left( k_{1}^{\alpha }+k_{2}^{\alpha }\right) \cos \left( r\sqrt{k}\right)
\cos \left( b\sqrt{k}\right) +\sin \left( r\sqrt{k}\right) \sin \left( b%
\sqrt{k}\right) =V\cos \left( \Theta _{1}-b\sqrt{k}\right) .
\end{equation*}%
On the other hand, as $\lambda _{1}+\lambda _{2}=1$, we have
\begin{eqnarray*}
&&k_{1}^{\alpha }\cos \left( \lambda _{1}a\sqrt{k}\right) +k_{2}^{\alpha
}\cos \left( \lambda _{2}a\sqrt{k}\right) \\
&=&k_{1}^{\alpha }\cos \left( \lambda _{1}a\sqrt{k}\right) +k_{2}^{\alpha
}\cos \left( a\sqrt{k}\right) \cos \left( \lambda _{1}a\sqrt{k}\right)
+k_{2}^{\alpha }\sin \left( a\sqrt{k}\right) \sin \left( \lambda _{1}a\sqrt{k%
}\right) \\
&=&\left( k_{1}^{\alpha }+k_{2}^{\alpha }\cos \left( a\sqrt{k}\right)
\right) \cos \left( \lambda _{1}a\sqrt{k}\right) +k_{2}^{\alpha }\sin \left(
a\sqrt{k}\right) \sin \left( \lambda _{1}a\sqrt{k}\right) \\
&=&W\cos \left( \Theta _{2}-\lambda _{1}a\sqrt{k}\right) ,
\end{eqnarray*}%
where
\begin{equation*}
We^{i\Theta _{2}}=\left( k_{1}^{\alpha }+k_{2}^{\alpha }\cos \left( a\sqrt{k}%
\right) \right) +i\left( k_{2}^{\alpha }\sin \left( a\sqrt{k}\right) \right)
=k_{1}^{\alpha }+k_{2}^{\alpha }e^{ia\sqrt{k}}
\end{equation*}%
as a complex number for some $\Theta _{2}\in \lbrack 0,2\pi )$. Thus,
inequality (\ref{V_W}) becomes
\begin{equation*}
V\cos \left( \Theta _{1}-b\sqrt{k}\right) \geq W\cos \left( \Theta
_{2}-\lambda _{1}a\sqrt{k}\right) .
\end{equation*}

Since $0<r\leq \frac{1}{2}D_{k}$, we have $0<a\leq 2r\leq \pi /\sqrt{k}$.
Then it is easy to see that $0<\Theta _{2}<a\sqrt{k}$. Let
\begin{equation}
\lambda _{1}=\frac{\Theta _{2}}{a\sqrt{k}}\in \left( 0,1\right) ,
\label{lambda}
\end{equation}%
we have the inequality $V\geq W$. That is,
\begin{equation*}
\left( k_{1}^{\alpha }+k_{2}^{\alpha }\right) ^{2}\cos ^{2}\left( r\sqrt{k}%
\right) +\sin ^{2}\left( r\sqrt{k}\right) \geq \left( k_{1}^{\alpha
}+k_{2}^{\alpha }\cos \left( a\sqrt{k}\right) \right) ^{2}+k_{2}^{2\alpha
}\sin ^{2}\left( a\sqrt{k}\right) .
\end{equation*}%
By simplifying this inequality, we get
\begin{equation*}
\cos \left( a\sqrt{k}\right) \leq 1-\frac{R^{2}}{2}\sin ^{2}\left( r\sqrt{k}%
\right) \text{ and thus }\sin \frac{a\sqrt{k}}{2}\geq \frac{R}{2}\sin \left(
r\sqrt{k}\right) .
\end{equation*}

The proof for the cases $k=0$ and $k<0$ are similar when using the ordinary
(or the hyperbolic) law of cosines in the model space $M_{k}^{2}$.
\end{proof}

Using lemma \ref{key_lemma}, we have the following upper bounds for $R$
defined as in (\ref{R_def}) which is useful when $\alpha <0$.

\begin{lemma}
\label{R_less_than_2}Let $R$ be defined as in (\ref{R_def}). For any $k$ and
$\alpha <1$, we have
\begin{equation*}
R\leq 2.
\end{equation*}
\end{lemma}

\begin{proof}
By the triangle inequality, we have $a\leq 2r\leq D_{k}$. We now use the
estimates in lemma \ref{key_lemma}.

When $k<0$, then
\begin{equation*}
\sinh \left( \sqrt{-k}r\right) \geq \sinh \frac{\sqrt{-k}a}{2}\geq \frac{R}{2%
}\sinh \left( \sqrt{-k}r\right) .
\end{equation*}%
This yields $R\leq 2$.

When $k=0$, then
\begin{equation*}
2r\geq a\geq Rr,
\end{equation*}%
so $R\leq 2.$

When $k>0$, then%
\begin{equation*}
\sin \left( r\sqrt{k}\right) \geq \sin \frac{a\sqrt{k}}{2}\geq \frac{R}{2}%
\sin \left( r\sqrt{k}\right)
\end{equation*}%
as $0\leq \frac{a\sqrt{k}}{2}\leq r\sqrt{k}\leq \frac{\pi }{2}$. Therefore,
we still have $R\leq 2$.
\end{proof}

The following proposition says that when $\alpha $ is negative, the weights
on any two directed edges from a common vertex of an $\alpha -$optimal
transport path are comparable to each other.

\begin{proposition}
If $\alpha <0$, then for each $i=1,2$,
\begin{equation*}
k_{i}\geq \frac{1}{1+\left( 1+2^{\alpha }\right) ^{-\frac{1}{\alpha }}},
\end{equation*}%
where $k_{i}$ is defined as in (\ref{k_i}).
\end{proposition}

\begin{proof}
Without losing generality, we may assume that $k_{2}\geq k_{1}$. By
proposition \ref{R_less_than_2}, we have $R\leq 2$. That is,
\begin{equation*}
\frac{\left( k_{1}^{\alpha }+k_{2}^{\alpha }\right) ^{2}-1}{k_{1}^{\alpha
}k_{2}^{\alpha }}\leq 4.
\end{equation*}%
Simplify it, we have%
\begin{equation*}
k_{1}^{\alpha }-k_{2}^{\alpha }\leq 1.
\end{equation*}%
Since $k_{1}\in (0,\frac{1}{2}]$ and $\alpha <0$, we have
\begin{equation*}
1-\left( \frac{k_{2}}{k_{1}}\right) ^{\alpha }\leq \left( k_{1}\right)
^{-\alpha }\leq 2^{\alpha }.
\end{equation*}%
Simplify it again using $k_{2}=1-k_{1}$, we have
\begin{equation*}
k_{1}\geq \frac{1}{1+\left( 1+2^{\alpha }\right) ^{-\frac{1}{\alpha }}}.
\end{equation*}
\end{proof}

We now may investigate the comparison angle $\theta $ between $A_{1}$ and $%
A_{2}$ at $O$, given in figure \ref{old_triangle}:

\begin{proposition}
Let $X$ be a geodesic metric space with curvature bounded above by a real
number $k$. Let $\theta $ be the comparison angle between $A_{1}$ and $A_{2}$
at $O$ in the model space $M_{k}^{2}$. Then
\begin{equation*}
\theta \geq \arccos \left( 1-\frac{R^{2}}{2}\right) =\arccos \left( \frac{%
1-k_{1}^{2\alpha }-k_{2}^{2\alpha }}{2k_{1}^{\alpha }k_{2}^{\alpha }}\right)
.
\end{equation*}%
Thus, by (\ref{R_a}), we have
\begin{equation*}
\theta \geq \theta _{\alpha }:=\left\{
\begin{array}{cc}
\frac{\pi }{2}, & \text{if }0<\alpha \leq \frac{1}{2} \\
\arccos \left( 2^{2\alpha -1}-1\right) , & \text{if }\frac{1}{2}<\alpha <1%
\text{ or }\alpha \leq 0%
\end{array}%
\right. .
\end{equation*}
\end{proposition}

Note that when $k=0$, this agrees with what we have found in \cite[Example
2.1]{xia1} for a ``Y-shaped'' path. Also, when $\alpha $ approaches $-\infty
$, then $\theta _{\alpha }$ approaches $\pi $, and when $\alpha $ approaches
$1$, then $\theta _{\alpha }$ approaches $0$.

\begin{proof}
When $k>0$, then by the spherical law of cosines,
\begin{eqnarray*}
\cos \theta &=&\frac{\cos \left( a\sqrt{k}\right) -\cos ^{2}\left( r\sqrt{k}%
\right) }{\sin ^{2}\left( r\sqrt{k}\right) } \\
&\leq &\frac{1-\frac{R^{2}}{2}\sin ^{2}\left( r\sqrt{k}\right) -\cos
^{2}\left( r\sqrt{k}\right) }{\sin ^{2}\left( r\sqrt{k}\right) }=1-\frac{%
R^{2}}{2}.
\end{eqnarray*}%
When $k<0$, then by the hyperbolic law of cosines
\begin{eqnarray*}
\cos \theta &=&\frac{-\cosh \left( a\sqrt{-k}\right) +\cosh ^{2}\left( r%
\sqrt{-k}\right) }{\sinh ^{2}\left( r\sqrt{-k}\right) } \\
&\leq &\frac{-1-\frac{R^{2}}{2}\sinh ^{2}\left( r\sqrt{-k}\right) +\cosh
^{2}\left( r\sqrt{-k}\right) }{\sinh ^{2}\left( r\sqrt{-k}\right) }=1-\frac{%
R^{2}}{2}.
\end{eqnarray*}%
When $k=0$, then by the law of cosines,
\begin{equation*}
\cos \theta =\frac{r^{2}+r^{2}-a^{2}}{2r^{2}}\leq \frac{2r^{2}-R^{2}r^{2}}{%
2r^{2}}=1-\frac{R^{2}}{2}.
\end{equation*}
\end{proof}

Now, we want to estimate the degree (i.e. the total number of edges) at each
vertex of an optimal transport path. We first rewrite lemma \ref{key_lemma}
as follows. For any real numbers $x\leq \left( \frac{\pi }{2}\right) ^{2}$
and $0<y\leq 2$, define
\begin{equation*}
\Psi \left( x,y\right) :=\left\{
\begin{tabular}{ll}
$\frac{1}{\sqrt{x}}\arcsin \left( \frac{y}{2}\sin \left( \sqrt{x}\right)
\right) ,$ & if $0<x\leq \left( \frac{\pi }{2}\right) ^{2}$ \\
$\frac{R}{2},$ & if $x=0$ \\
$\frac{1}{\sqrt{-x}}\sinh ^{-1}\left( \frac{y}{2}\sinh \left( \sqrt{-x}%
\right) \right) $ & if $x<0$%
\end{tabular}%
\right. .
\end{equation*}%
Then, one may check that $\Psi $ is a continuous strictly decreasing
function of the variable $x$ and an increasing function of $y$. Moreover,
for each fixed $y$, $\lim_{x\rightarrow -\infty }\Psi \left( x,y\right) =1$
and
\begin{equation*}
\Psi \left( x,y\right) \geq \text{ }\Psi \left( \left( \frac{\pi }{2}\right)
^{2},y\right) =\frac{2}{\pi }\arcsin \left( \frac{y}{2}\right) .
\end{equation*}

By means of the function $\Psi $, the lemma \ref{key_lemma} becomes

\begin{lemma}
\label{key_lemma_2}Assume that $d\left( O,A_{1}\right) =d\left(
O,A_{2}\right) =r$ and $0<r\leq \frac{1}{2}D_{k}$. Then, we have the
following estimate for $d\left( A_{1},A_{2}\right) $:
\begin{equation*}
2r\geq d\left( A_{1},A_{2}\right) \geq 2r\Psi \left( r^{2}k,R\right) \geq
2rC_{\alpha }
\end{equation*}%
where $C_{\alpha }:=\frac{2}{\pi }\arcsin \left( \frac{R_{\alpha }}{2}%
\right) $.
\end{lemma}

Note that since $\lim_{x\rightarrow -\infty }\Psi \left( x,R\right) =1$, we
have $d\left( A_{1},A_{2}\right) $ is nearly $2r$ when $k$ approaches $%
-\infty $.

Let
\begin{equation}
\Phi \left( x,\alpha \right) =1+\frac{\ln \left( 1+\frac{1}{\Psi \left(
x,R_{\alpha }\right) }\right) }{\ln 2}  \label{Phi}
\end{equation}%
for $x\in \mathbb{R}$ and $\alpha <1$. For each fixed $\alpha <1$, $\Phi
_{\alpha }\left( k\right) :=$ $\Phi \left( k,\alpha \right) $ is a strictly
increasing function of $k$ with lower bound $\lim_{k\rightarrow -\infty
}\Phi \left( k,\alpha \right) =2$, upper bound
\begin{equation*}
\Phi \left( k,\alpha \right) \leq 1+\frac{\ln \left( 1+\frac{1}{C_{\alpha }}%
\right) }{\ln 2}
\end{equation*}%
and
\begin{equation*}
\Phi \left( 0,\alpha \right) =1+\frac{\ln \left( 1+\frac{2}{R_{\alpha }}%
\right) }{\ln 2}.
\end{equation*}

As in \cite[10.13]{heinonen}, a metric space $X$ is called \textit{doubling}
if there is a constant $C_{d}\geq 1$ so that every subset of diameter $r$ in
$X$ can be covered by at most $C_{d}$ subsets of diameter at most $\frac{r}{2%
}$. Doubling spaces have the following covering property: there exists
constants $\beta >0$ and $C_{\beta }\geq 1$ such that for every $\epsilon
\in (0,\frac{1}{2}]$, every set of diameter $r$ in $X$ can be covered by at
most $C_{\beta }\epsilon ^{-\beta }$ sets of diameter at most $\epsilon r$.
This function $C_{\beta }\epsilon ^{-\beta }$ is called a covering function
of $X$. The infimum of all numbers $\beta >0$ such that a covering function
can be found is called the \textit{Assouad dimension} of $X$. It is clear
that subsets of doubling spaces are still doubling. For any subset $K$ of $X$%
, let $\dim _{A}\left( K\right) $ denote the Assouad dimension of $K$.

\begin{theorem}
Suppose $X$ is a geodesic doubling metric space of curvature bounded above
by a real number $k$. Let $\alpha <1$ and $G$ be an $\alpha -$optimal
transport path between two atomic probability measures on $X$, and $O$ is a
vertex of $G$. Let $\deg \left( O\right) $ be the degree of the vertex $O$
and $r\left( O\right) $ be the maximum number $r$ in $(0,\frac{1}{2}D_{k}]$
such that the truncated ball $B\left( O,r\right) \backslash \left\{
O\right\} $ contains no vertices of $G$. Then,
\end{theorem}

\begin{enumerate}
\item for any $0<r\leq r\left( O\right) $, we have
\begin{equation*}
\deg \left( O\right) \leq 2\left( C_{d}\right) ^{\Phi \left( r^{2}k,\alpha
\right) },
\end{equation*}%
where $C_{d}$ is the doubling constant of $X$, and $\Phi $ is given in (\ref%
{Phi}).

\item Moreover, $\deg \left( O\right) \leq 2\left( C_{d}\right) ^{\Phi
\left( 0,\alpha \right) }$, which is a constant depends only on $\alpha $
and $C_{d}$.

\item If $\deg \left( O\right) \geq 2\left( C_{d}\right) ^{2}$, then the
curvature upper bound
\begin{equation*}
k\geq \frac{1}{r\left( O\right) ^{2}}\left( \Phi _{\alpha }\right)
^{-1}\left( \log _{C_{d}}^{\deg \left( O\right) /2}\right) .
\end{equation*}

\item In particular, if $\deg \left( O\right) =2\left( C_{d}\right) ^{\Phi
\left( 0,\alpha \right) }$, then $k\geq 0$.

\item If $k<0$, then
\begin{equation*}
r\left( O\right) \leq \sqrt{\frac{\left( \Phi _{\alpha }\right) ^{-1}\left(
\log _{C_{d}}^{\deg \left( O\right) /2}\right) }{k}}.
\end{equation*}
\end{enumerate}

\begin{proof}
For any $0<r\leq r\left( O\right) $, let $\left\{ A_{i}\right\} $ be the
intersection points of the sphere $S\left( O,r\right) $ in $X$ with all
edges of $G$ that flows out of $O$. It is sufficient to show that the
cardinality of $\left\{ A_{i}\right\} $ is bounded above by $\left(
C_{d}\right) ^{\Phi \left( r^{2}k,\alpha \right) }$. By lemma \ref%
{key_lemma_2}, $\left\{ B\left( A_{i},r\Psi \left( r^{2}k,R_{\alpha }\right)
\right) \right\} $ are disjoint and contained in $B\left( O,\left( 1+\Psi
\left( r^{2}k,R_{\alpha }\right) \right) r\right) $. Since $X$ is doubling,
the cardinality of $\left\{ B\left( A_{i},\Psi \left( r^{2}k,R_{\alpha
}\right) r\right) \right\} $ is bounded above by $\left( C_{d}\right) ^{\Phi
\left( r^{2}k,\alpha \right) }$. This proves $(1)$. By setting $r\rightarrow
0$ in $\left( 1\right) $, we have $\left( 2\right) $. Then (3) and $\left(
5\right) $ follow from $\left( 1\right) $, and $\left( 4\right) $ follows
from $\left( 3\right) $.
\end{proof}

\section{Optimal transport paths between arbitrary probability measures}

In this section, we consider optimal transport paths between two arbitrary
probability measures on a complete geodesic metric space $\left( X,d\right)$%
. Unlike what we did in Euclidean space \cite{xia1}, we will use a new
approach by considering the completion of $\mathcal{A}\left( X\right) $ with
respect to the metric $d_{\alpha }$. Note that $\left( \mathcal{A}\left(
X\right) ,d_{\alpha }\right) $ is not necessarily complete for $\alpha <1$.
So, we consider its completion as follows.

\begin{definition}
For any $\alpha \in (-\infty ,1]$, let $\mathcal{P}_{\alpha }(X)$ be the
completion of the metric space $\mathcal{A}(X)$ with respect to the metric $%
d_{\alpha }$.
\end{definition}

It is easy to check that (see \cite[lemma 2.2.5]{xia6}) if $\beta <\alpha $,
then $\mathcal{P}_{\beta }(X)\subseteq \mathcal{P}_{\alpha }(X)$, and for
all $\mu ,\nu $ in $\mathcal{P}_{\beta }(X)$ we have $d_{\beta }(\mu ,\nu
)\geq d_{\alpha }(\mu ,\nu )$. Note that when $\alpha =1$, the metric $d_{1}$
is the usual Monge's distance on $\mathcal{A}(X)$ and $\mathcal{P}_{1}\left(
X\right) $ is just the space $\mathcal{P}\left( X\right) $ of all
probability measures on $X$. Therefore, each element in $\mathcal{P}_{\alpha
}$ can be viewed as a probability measure on $X$ when $\alpha <1$.

By proposition \ref{all_atomic_measures}, the concept of an $\alpha -$%
optimal transport path on $\mathcal{A}(X)$ coincides with the concept of
geodesic in $\left( \mathcal{A}(X),d_{\alpha }\right) $. This motivates us
to introduce the following concept.

\begin{definition}
For any two probability measures $\mu ^{+}$ and $\mu ^{-}$ on a complete
geodesic metric space $X$ and $\alpha<1$, if there exists a geodesic in $%
\left( \mathcal{P}_{\alpha }\left( X\right) ,d_{\alpha }\right) $ from $\mu
^{+}$ to $\mu ^{-}$, then this geodesic is called an $\alpha -$optimal
transport path from $\mu ^{+}$ to $\mu ^{-}$.
\end{definition}

In other words, the existence of an $\alpha -$optimal transport path is the
same as the existence of a geodesic in $\mathcal{P}_{\alpha }(\mathbf{{X})}$%
. Thus, an essential part in understanding the optimal transport problem
becomes describing properties of elements of $\mathcal{P}_{\alpha }(\mathbf{{%
X})}$, and investigating the existence of geodesics in $\mathcal{P}_{\alpha
}(\mathbf{{X})}$. Since completion of a geodesic metric space is still a
geodesic space, by proposition \ref{all_atomic_measures}, we have

\begin{proposition}
\label{p_alpha_geodesic}Suppose $X$ is a complete geodesic metric space.
Then for any $0\leq \alpha <1$, $\left( \mathcal{P}_{\alpha }\left( X\right)
,d_{\alpha }\right) $ is a complete geodesic metric space.
\end{proposition}

In other words, for any two probability measures $\mu ^{+},\mu ^{-}\in
\mathcal{P}_{\alpha }(X)$ with $0\leq \alpha <1$, there exists an optimal
transport path (i.e. a geodesic) from $\mu ^{+}$ to $\mu ^{-}$. In
particular, since atomic measures are contained in $\mathcal{P}_{\alpha }(X)$%
, there exists an $\alpha -$optimal transport path from any probability
measure $\mu \in \mathcal{P}_{\alpha }(X)$ to $\delta _{p}$ for any $p\in X$.

A positive Borel measure $\mu $ on $X$ is said to be \textit{concentrated}
on a Borel set $A$ if $\mu (X\setminus A)=0$. The following proposition says
that if $\alpha $ is nonpositive, then any element of $\mathcal{P}_{\alpha
}(X)$ must be bounded.

\begin{proposition}
Suppose $\alpha \leq 0$. If $\mu \in \mathcal{P}_{\alpha }(X)$, then $\mu $
is concentrated on the closed ball $\bar{B}\left( p,d_{\alpha }\left( \mu
,\delta _{p}\right) \right) $ for any $p\in X$.
\end{proposition}

\begin{proof}
If $\mu \in \mathcal{P}_{\alpha }(X)$, then $\mu $ is represented by a
Cauchy sequence $\left\{ \mathbf{a}_{n}\right\} \in \mathcal{A}(X)$ with
respect to the metric $d_{\alpha }$. For any $p\in X,$ by corollary \ref%
{negative_estimate}, each $\mathbf{a}_{n}$ is concentrated on the ball $\bar{%
B}\left( p,d_{\alpha }\left( \mathbf{a}_{n},\delta _{p}\right) \right) $.
Thus, $\mu $ is concentrated on the ball $\bar{B}\left( p,d_{\alpha }\left(
\mu ,\delta _{p}\right) \right).$
\end{proof}

When $0<\alpha <1$, $\mu \in \mathcal{P}_{\alpha }(X)$ does not necessarily
imply $\mu $ is concentrated on a bounded set. For instance, let $X=\mathbb{R%
}$, and $\mu =\sum_{n=1}^{\infty }\frac{1}{2^{n}}\delta _{\left\{ n\right\}
} $, which is clearly unbounded. Then, $d_{\alpha }\left( \mu ,\delta
_{\left\{ 0\right\} }\right) =\sum_{n=1}^{\infty }\left( \sum_{k=n}^{\infty }%
\frac{1}{2^{k}}\right) ^{\alpha }\cdot 1=\sum_{n=1}^{\infty }\left( \frac{1}{%
2^{n-1}}\right) ^{\alpha }<\infty $, and $\mu \in \mathcal{P}_{\alpha }(X)$.
Nevertheless, the following proposition says that the mass of every measure
in $\mathcal{P}_{\alpha }(X)$ outside a ball decays to $0$ as the radius of
the ball increases.

\begin{proposition}
\label{general_positive_estimate}Suppose $0<\alpha <1$ and $\tilde{\mu}%
=\lambda \mu $ for some $\mu \in \mathcal{P}_{\alpha }\left( X\right) $ and $%
\lambda >0$. Then, for any point $p\in X$ and $r>0$, we have
\begin{equation*}
\left[ \tilde{\mu}(X\setminus \bar{B}(p,r))\right] ^{\alpha }\leq \frac{%
d_{\alpha }(\tilde{\mu},\lambda \delta _{p})}{r}.
\end{equation*}%
In particular, if $r\geq d_{\alpha }(\tilde{\mu},\lambda \delta
_{p})^{1-\alpha }$, we have
\begin{equation}
\tilde{\mu}(X\setminus \bar{B}(p,r))\leq d_{\alpha }(\tilde{\mu},\lambda
\delta _{p}).  \label{mu_d_alpha}
\end{equation}
\end{proposition}

\begin{proof}
By corollary \ref{positive_estimate},%
\begin{equation*}
\left[ \mu (X\setminus \bar{B}(p,r))\right] ^{\alpha }\leq \frac{d_{\alpha
}(\mu ,\delta _{p})}{r}.
\end{equation*}%
Now, for any $\lambda >0$,%
\begin{equation*}
\left[ \tilde{\mu}(X\setminus \bar{B}(p,r))\right] ^{\alpha }=\left[ \lambda
\mu (X\setminus \bar{B}(p,r))\right] ^{\alpha }\leq \lambda ^{\alpha }\frac{%
d_{\alpha }(\mu ,\delta _{p})}{r}=\frac{d_{\alpha }(\tilde{\mu},\lambda
\delta _{p})}{r}.
\end{equation*}
\end{proof}

\section{Transport Dimension of measures on a metric space}

Now, a natural question is to describe properties of measures that lie in
the space $\mathcal{P}_{\alpha }(X)$. The answer to this question crucially
related to the dimensional information of the measure $\mu $. So, we study
the dimension of measures in this section. To do it, we first study
properties of measures belong to a special subset of $\mathcal{P}_{\alpha
}\left( X\right) $, and then using it to define the transport dimension of
measures. The work in this section generalize results of \cite{xia6} from
Euclidean spaces to a complete geodesic metric space, while the study of %
\cite{xia6} is motivated by the work of \cite{Solimini}.

\subsection{$d_{\protect\alpha }-$\textbf{admissible Cauchy sequence}}

\begin{definition}
\label{admissible}Suppose $X$ is a complete geodesic metric space. Let $\{%
\mathbf{a}_{k}\}_{k=1}^{\infty }$be a sequence of atomic measures on $X$ of
equal total mass in the form of
\begin{equation*}
\mathbf{a}_{k}=\sum_{i=1}^{N_{k}}m_{i}^{\left( k\right) }\delta
_{x_{i}^{\left( k\right) }}
\end{equation*}%
for each $k$, and $\alpha <1$. We say that this sequence is a $d_{\alpha }-$%
\textbf{admissible Cauchy sequence} if for any $\epsilon >0$, there exists
an $N$ such that for all $n>k\geq N$ there exists a partition of
\begin{equation*}
\mathbf{a}_{n}=\sum_{i=1}^{N_{k}}\mathbf{a}_{n,i}^{(k)}
\end{equation*}%
with respect to $\mathbf{a}_{k}$ as sums of disjoint atomic measures and a
path (see figure \ref{admissible_path})
\begin{equation*}
G_{n,i}^{k}\in Path(m_{i}^{(k)}\delta _{x_{i}^{(k)}},\mathbf{a}_{n,i}^{(k)})
\end{equation*}%
for each $i=1,2,\cdots ,N_{k}$ such that
\begin{equation*}
\sum_{i=1}^{N_{k}}\mathbf{M}_{\alpha }\left( G_{n,i}^{k}\right) \leq
\epsilon \text{.}
\end{equation*}%
Also, we denote $G_{n}^{k}=\sum_{i=1}^{N_{k}}G_{n,i}^{k}$, which is a path
from $\mathbf{a}_{k}$ to $\mathbf{a}_{n}$ with $\mathbf{M}_{\alpha
}(G_{n}^{k})\leq \epsilon $. Each $d_{\alpha }-$\textbf{admissible Cauchy
sequence }corresponds to an element in $\mathcal{P}_{\alpha }(X)$. Let
\begin{equation*}
\mathcal{D}_{\alpha }(X)\subseteq \mathcal{P}_{\alpha }(X)
\end{equation*}%
be the set of all probability measures $\mu $ which corresponds to a $%
d_{\alpha }$ admissible Cauchy sequence of probability measures. For
simplicity, we may write $\mathcal{D}_{\alpha }(X)$ as $\mathcal{D}_{\alpha
} $.
\end{definition}

\begin{figure}[h!]
\caption{An example of a transport path between $a_k$ and $a_n$.}
\label{admissible_path}\centering
\includegraphics[height=1.8in]{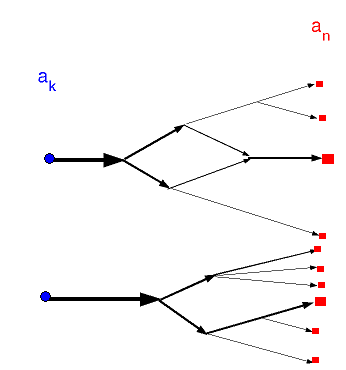}
\end{figure}

It is easy to see that if for each $k$, there is a partition of $\mathbf{a}%
_{k+1}=\sum_{i=1}^{N_{k}}\mathbf{a}_{k+1,i}^{(k)}$ with respect to $\mathbf{a%
}_{k}$ as sums of disjoint atomic measures and a path $G_{k+1,i}^{k}\in
Path\left( m_{i}^{(k)}\delta _{x_{i}^{(k)}},\mathbf{a}_{k+1,i}^{(k)}\right) $
for each $i=1,2,\cdots ,N_{k}$ such that
\begin{equation*}
\sum_{k=1}^{\infty }\left( \sum_{i=1}^{N_{k}}\mathbf{M}_{\alpha }\left(
G_{k+1,i}^{k}\right) \right) <+\infty ,
\end{equation*}%
then $\left\{ \mathbf{a}_{n}\right\} $ is a $d_{\alpha }$-admissible Cauchy
sequence.

Also, note that if $\mu ,\nu \in \mathcal{D}_{\alpha }(X)$, one
automatically has $d_{\alpha }\left( \mu ,\nu \right) <\infty $.

\subsection{Relation with Hausdorff dimension of measures}

Let $\mathcal{H}^{s}$ denote $s$ dimensional Hausdorff measure on $X$ for
each $s\geq 0$. By means of corollary \ref{negative_estimate} and (\ref%
{mu_d_alpha}), the proof of \cite[theorem 3.2.1]{xia6} is also valid for the
following theorem:

\begin{theorem}
\label{1.2} Suppose $X$ is a complete geodesic metric space. If $\mu \in
\mathcal{D}_{\alpha }\left( X\right) $ for some $\alpha \in \left( -\infty
,1\right) $, then $\mu $ is concentrated on a subset $A$ of $X$ with $%
\mathcal{H}^{\frac{1}{1-\alpha }}\left( A\right) =0$.
\end{theorem}

\begin{definition}
For any probability measure $\mu $ on a complete geodesic metric space $X$,
the Hausdorff dimension of $\mu $ is defined to be
\begin{equation*}
\dim _{H}\left( \mu \right) =\inf \left\{ \dim _{H}\left( A\right) :\mu
\left( X\backslash A\right) =0\right\} ,
\end{equation*}%
where $\dim _{H}(A)$ is the Hausdorff dimension of a set $A$.
\end{definition}

Thus, by Theorem \ref{1.2}, we have

\begin{corollary}
Suppose $X$ is a complete geodesic metric space. For any $\alpha <1$ and any
$\mu \in \mathcal{D}_{\alpha }(X)$, we have
\begin{equation*}
\dim _{H}(\mu )\leq \frac{1}{1-\alpha }.
\end{equation*}
\end{corollary}

\subsection{Minkowski dimension of measures}

A \textit{nested collection}
\begin{equation}
\mathcal{F}=\left\{ Q_{i}^{n}:i=1,2,\cdots ,N_{n}\text{ and }n=1,2,\cdots
\right\}  \label{nested_collection}
\end{equation}%
\textit{of cubes} in $X$ is a collection of Borel subsets of $X$ with the
following properties:

\begin{enumerate}
\item for each $Q_{i}^{n}$, its diameter
\begin{equation}
C_{1}\sigma ^{n}\leq diam\left( Q_{i}^{n}\right) \leq C_{2}\sigma ^{n}
\label{condition_1}
\end{equation}%
for some constants $C_{2}\geq C_{1}>0$ and some $\sigma \in \left(
0,1\right) $.

\item for any $k,l,i,j$ with $l\geq k$, either $Q_{i}^{k}\cap
Q_{j}^{l}=\emptyset $ or $Q_{i}^{k}\subseteq Q_{j}^{l};$

\item for each $Q_{j}^{n+1}$ there exists exactly one $Q_{i}^{n}$ (parent of
$Q_{j}^{n+1}$) such that $Q_{j}^{n+1}\subseteq Q_{i}^{n}$;

\item for each $Q_{i}^{n}$ there exists at least one $Q_{j}^{n+1}$ (child of
$Q_{i}^{n}$) such that $Q_{j}^{n+1}\subseteq Q_{i}^{n}$;
\end{enumerate}

Each $Q_{i}^{n}$ is called a cube of generation $n$ in $\mathcal{F}$. If two
different cubes $Q_{i}^{n}$ and $Q_{j}^{n}$ of generation $n$ have the same
parent, then they are called \textit{brothers} to each other.

In next section, we will see that for each bounded subset of a complete
geodesic doubling metric space, there always exists a nested collection of
cubes which covers the set.

\begin{definition}
For any nested collection $\mathcal{F}$, we define its Minkowski dimension
\begin{equation}
\dim _{M}\left( \mathcal{F}\right) :=\lim_{n\rightarrow \infty }\frac{\log
\left( N_{n}\right) }{\log \left( \frac{1}{\sigma ^{n}}\right) }
\label{cubical_dim}
\end{equation}%
provided the limit exists, where $N_{n}$ is the total number of cubes of
generation $n$.
\end{definition}

\begin{definition}
A Radon measure $\mu $ on $X$ is said to be \textit{concentrated} on a
nested collection $\mathcal{F}$ in $X$ if for each $n$,
\begin{equation*}
\mu \left( X\setminus \left( \bigcup_{i=1}^{N_{n}}Q_{i}^{n}\right) \right)
=0.
\end{equation*}
\end{definition}

\begin{definition}
For any Radon measure $\mu $, we define the Minkowski dimension of the
measure $\mu $ to be
\begin{equation*}
\dim _{M}\left( \mu \right) :=\inf \left\{ \dim _{M}\left( \mathcal{F}%
\right) \right\}
\end{equation*}%
where the infimum is over all nested collection $\mathcal{F}$ that $\mu $ is
concentrated on.
\end{definition}

The proof of \cite[theorem 3.3.5]{xia6} is still valid for the following
theorem:

\begin{theorem}
\label{positive_minkowski} Suppose $\mu $ is a probability measure on a
complete geodesic metric space $\left( X,d\right) $. If $\dim _{M}(\mu )<%
\frac{1}{1-\alpha }$ for some $0\leq \alpha <1$, then $\mu \in \mathcal{D}%
_{\alpha }\left( X\right) $.
\end{theorem}

\subsection{Evenly concentrated measures}

Now, we aim at achieving a similar result as in theorem \ref%
{positive_minkowski} for the case $\alpha <0$. To do it, we introduce the
following definition:

\begin{definition}
Let $\left( X,d\right) $ be a complete geodesic metric space. A Radon
measure $\mu $ on $X$ is \textit{evenly concentrated} on a nested collection
$\mathcal{F}$ in $X$ if for each cube $Q_{i}^{n}$ of generation $n$ in $%
\mathcal{F}$, either $Q_{i}^{n}$ has no brothers or $\mu \left(
Q_{i}^{n}\right) \geq \frac{\lambda }{N_{n}}$ for some constant $\lambda >0$.
\end{definition}

Here, $Q_{i}^{n}$ has no brothers means that the parent of $Q_{i}^{n}$ has
only one child, namely $Q_{i}^{n}$ itself.

Some examples of evenly concentrated measures have been given in \cite{xia6}%
. In particular, if $\mu $ is an Ahlfors regular measure concentrated on a
nested collection $\mathcal{F}$ in $X$, then $\mu $ is evenly concentrated
on $\mathcal{F}$.

\begin{definition}
For any Radon measure $\mu $, we define
\begin{equation*}
\dim _{U}\left( \mu \right) :=\inf \left\{ \dim _{M}\left( \mathcal{F}%
\right) \right\}
\end{equation*}%
where the infimum is over all nested collection $\mathcal{F}$ that $\mu $ is
evenly concentrated on.
\end{definition}

Obviously,
\begin{equation*}
\dim _{M}\left( \mu \right) \leq \dim _{U}\left( \mu \right) .
\end{equation*}%
The proof of \cite[theorem 3.4.6]{xia6} is also valid for the following
theorem:

\begin{theorem}
\label{negative}Let $\left( X,d\right) $ be a complete geodesic metric
space. Suppose $\mu $ is a probability measure with $\dim _{U}(\mu )<\frac{1%
}{1-\alpha }$ for some $\alpha <1$, then $\mu \in \mathcal{D}_{\alpha
}\left( X\right) $.
\end{theorem}

\subsection{Transport dimension of measures}

We now introduce the following concept:

\begin{definition}
Suppose $X$ is a complete geodesic metric space. For any probability measure
$\mu $ on $X$, we define the transport dimension of $\mu $ to be
\begin{equation*}
\dim _{T}\left( \mu \right) :=\inf_{\alpha <1}\left\{ \frac{1}{1-\alpha }%
:\mu \in \mathcal{D}_{\alpha }(X)\right\} .
\end{equation*}
\end{definition}

Note that if $\frac{1}{1-\alpha }>\dim _{T}\left( \mu \right) $, then $\mu
\in \mathcal{D}_{\alpha }(X)$, and thus $d_{\alpha }\left( \mu ,\delta
_{O}\right) <+\infty $ for any fixed point $O\in X$. If in addition $\alpha
\geq 0$, then there exists an $\alpha -$optimal transport path from $\mu $
to $\delta _{O}$.

By theorems \ref{positive_minkowski}, \ref{negative} and \ref{1.2}, we have
(see \cite[theorem 3.5.2]{xia6})

\begin{theorem}
\label{main theorem} Suppose $X$ is a complete geodesic metric space. Let $%
\mu $ be any probability measure on $X$, then
\begin{equation*}
\dim _{H}\left( \mu \right) \leq \dim _{T}\left( \mu \right) \leq \max
\{\dim _{M}(\mu ),1\}\text{.}
\end{equation*}%
Moreover, we also have%
\begin{equation*}
\dim _{H}\left( \mu \right) \leq \dim _{T}(\mu )\leq \dim _{U}(\mu ).
\end{equation*}
\end{theorem}

In \cite[Example 3.5.3]{xia6}, we showed that for the Cantor measure $\mu $,
we have
\begin{equation*}
\dim _{H}\left( \mu \right) =\dim _{T}(\mu )=\dim _{U}(\mu )=\frac{\ln 2}{%
\ln 3}.
\end{equation*}

\subsection{The Dimensional Distance between probability measures}

In this subsetion, we will give a geometric meaning to the transport
dimension of measures.

Let $(X,d)$ be a complete geodesic metric space. For any $\alpha <1$, let
\begin{equation*}
\mathcal{S}_{\alpha }\left( X\right) =\left\{ \Lambda \left( \mu -\nu
\right) :\Lambda \geq 0,\mu ,\nu \in \mathcal{D}_{\alpha }(X)\right\}
\end{equation*}%
be a collection of signed measures. Clearly, $\mathcal{S}_{\alpha
_{1}}\left( X\right) \subseteq \mathcal{S}_{\alpha _{2}}\left( X\right) $ if
$\alpha _{1}\leq \alpha _{2}$.

\begin{definition}
\label{Def_D} Let $(X,d)$ be a complete geodesic metric space. For any two
probability measures $\mu ,\nu $ on $X$, we define%
\begin{equation*}
D(\mu ,\nu ):=\inf_{\alpha <1}\{\frac{1}{1-\alpha }:\mu -\nu \in \mathcal{S}%
_{\alpha }\left( X\right) \}.
\end{equation*}
\end{definition}

The proof of \cite[proposition 4.05]{xia6} still valid for the following
proposition

\begin{proposition}
\label{pseudometric} Let $(X,d)$ be a complete geodesic metric space. Then, $%
D\,$ is a pseudometric\footnote[1]{%
A pseudometric $D$ means that it is nonnegative, symmetric, satisfies the
triangle inequality, and $D\left( \mu ,\mu \right) =0$. But $D\left( \mu
,\nu \right) =0$ does not imply $\mu =\nu $.} on the space of probability
measures on $X$.
\end{proposition}

In general, $D$ is not necessarily a metric. Indeed, for any two atomic
probability measures $\mathbf{a,b}$, we have $\mathbf{a-b}\in \mathcal{S}%
_{\alpha }\left( X\right) $ for any $\alpha <1$. Thus, $D\left( \mathbf{a,b}%
\right) =0$ while $\mathbf{a}$ and $\mathbf{b}$ are not necessarily the same
measure. Nevertheless, we may easily extend the pseudometric $D$ to a metric
on equivalent classes of measures. To this end we define a notion of the
equivalent class on measures.

\begin{definition}
For any two probability measures $\mu $ and $\nu $ on $X$, we say
\begin{equation*}
\mu \sim \nu \text{\ if }D(\mu ,\nu )=0.
\end{equation*}%
The equivalent class of $\mu $ is denoted by $[\mu ]$.
\end{definition}

For instance, all atomic probability measures are equivalent to each other.

\begin{definition}
For any equivalent class $[\mu]$ and $[\nu]$, define
\begin{equation*}
\mathbf{D}([\mu],[\nu])=D(\mu,\nu).
\end{equation*}
\end{definition}

From this definition and proposition \ref{pseudometric}, clearly, we have
the following theorem.

\begin{theorem}
Let $(X,d)$ be a complete geodesic metric space. Then, $\mathbf{D}$ is a
metric on $\mathcal{P}\left( X\right) /\sim $ .
\end{theorem}

\begin{definition}
The metric $\mathbf{D}$ is called the dimensional distance on the space $%
\mathcal{P}\left( X\right) /\sim $ \ of equivalent classes of probability
measures on $X$.
\end{definition}

We now give a geometric meaning to transport dimension of measures.

\begin{theorem}
Let $(X,d)$ be a complete geodesic metric space. For any positive
probability measure $\mu $, we have
\begin{equation*}
\dim _{T}(\mu )=\mathbf{D}\left( \left[ \mu \right] ,\left[ \mathbf{a}\right]
\right) =D(\mu ,\mathbf{a})
\end{equation*}%
where $\mathbf{a}$ is any atomic probability measure.
\end{theorem}

\begin{proof}
Since $\mathbf{a\in }\mathcal{D}_{\alpha }\left( X\right) $ for any $\alpha
<1$, we have
\begin{eqnarray*}
D(\mu ,\mathbf{a}) &=&\inf_{\alpha <1}\{\frac{1}{1-\alpha }:\mu -\mathbf{a}%
\in S_{\alpha }\} \\
&=&\inf_{\alpha <1}\{\frac{1}{1-\alpha }:\mu \in \mathcal{D}_{\alpha }\left(
X\right) \}=\dim _{T}\left( \mu \right).
\end{eqnarray*}
\end{proof}

This theorem says that the transport dimension of a probability measure $\mu
$ is the distance from $\mu $ to any atomic measure with respect to the
dimensional distance. In other words, the dimension information of a measure
tells us quantitatively how far the measure is from being an atomic measure.

\section{Measures on a complete doubling metric space}

In this section, we will show that there exists an $\alpha -$optimal
transport path between any two probability measures on a compact doubling
geodesic metric space $X$ whenever $\max \left\{ 1-\frac{1}{m},0\right\}
<\alpha <1$, where $m$ is the Assouad dimension of $X$.

Recall that a \textit{space of homogeneous type} (\cite{coifman}) is a
quasimetric space\ $X$ equipped with a doubling measure $\nu $, which is a
Radon measure on $X$ satisfying
\begin{equation*}
\nu \left( B\left( x,2r\right) \right) \leq C\nu \left( B\left( x,r\right)
\right)
\end{equation*}%
for any ball $B\left( x,r\right) $ in $X$ and for some constant
$C>0$. When $(X,d)$ is a metric space equipped with a doubling
measure $\nu $, then the triple $(X,d,\nu)$ is called a metric
measures space. Recently, many works (see \cite{cheeger},
\cite{heinonen}, etc) have been done on studying analysis on metric
measure spaces, in particular, when the measure $\mu$ is doubling
and satisfying the Poincaré inequality.

In %
\cite{christ6} and \cite{christ7}, Christ introduced a decomposition of a
space of homogeneous type as cubes and proved the following proposition:

\begin{proposition}
Suppose $\left( X,\nu \right) $ is a space of homogeneous type. For any $%
k\in \mathbb{Z}$, there exists a set, at most countable $I_{k}$ and a family
of subsets $Q_{\theta }^{k}\subseteq X$ with $\theta \in I_{k}$, such that
\end{proposition}

\begin{enumerate}
\item $\nu \left( X\setminus \cup _{\theta }Q_{\theta }^{k}\right) =0$, $%
\forall k\in \mathbb{Z}$;

\item for any $k,l,\theta ,\eta $ with $l\leq k$, either $Q_{\theta
}^{k}\cap Q_{\eta }^{l}=\emptyset $ or $Q_{\theta }^{k}\subseteq Q_{\eta
}^{l};$

\item for each $Q_{\eta }^{n+1}$ there exists exactly one $Q_{\theta }^{n}$
(parent of $Q_{\eta }^{n+1}$) such that $Q_{\eta }^{n+1}\subseteq Q_{\theta
}^{n}$;

\item for each $Q_{\theta }^{n}$ there exists at least one $Q_{\eta }^{n+1}$
(child of $Q_{\theta }^{n}$) such that $Q_{\eta }^{n+1}\subseteq Q_{\theta
}^{n}$;
\end{enumerate}

These open subsets of the kind $Q_{\theta }^{k}$ are called dyadic
cubes of generation $k$ due to the analogous between them and the
standard Euclidean dyadic cubes. A useful property regarding such
dyadic cubes is: there exists a point $x_{\theta }^{k}\in X$ for
each cube $Q_{\theta }^{k}$ such that
\begin{equation*}
B\left( x_{\theta }^{k},C_{0}\sigma ^{k}\right) \subseteq Q_{\theta
}^{k}\subseteq B\left( x_{\theta }^{k},C_{1}\sigma ^{k}\right)
\end{equation*}%
for some constants $C_{0},C_{1}$ and $\sigma \in \left( 0,1\right) $.
Moreover, for any $x_{\theta }^{k}$ and $x_{\eta }^{k}$, $d\left( x_{\theta
}^{k},x_{\eta }^{k}\right) \geq \sigma ^{k}$.

Since $\nu $ is a doubling measure, from $\left( 1\right) $, we see that $%
X=\cup _{\theta }\bar{Q}_{\theta }^{k}$, where $\bar{Q}_{\theta }^{k}$
denotes the closure of $Q_{\theta }^{k}$. Then, it is easy to see that there
exists a family of Borel subsets $\left\{ B_{\theta }^{k}\right\} $ with $%
\theta \in I_{k}$ such that $Q_{\theta }^{k}\subseteq B_{\theta
}^{k}\subseteq \bar{Q}_{\theta }^{k}$ with $X=\cup _{\theta }B_{\theta }^{k}$
for each $k$, and $\left\{ B_{\theta }^{k}\right\} $ still satisfy
conditions (2,3,4) above.

A very useful fact is pointed out in \cite[theorem 13.3]{heinonen}: every
complete doubling metric space $\left( X,d\right) $ has a nontrivial
doubling measure on it. Thus, one may also construct a family of disjoint
Borel subsets $\left\{ B_{\theta }^{k}\right\} $ for $\left( X,d\right) $ as
above. Now, for any bounded subset $K$ of $X$, we set $\mathcal{F}_{K}$ to
be the collection of all dyadic cubes $B_{\theta }^{k}$ that has a nonempty
intersection with $K$. It is easy to check that $\mathcal{F}_{K}$ is a
nested collection of cubes as defined in (\ref{nested_collection}).
Moreover, for any $\beta >\dim _{A}\left( K\right) $, from the definition of
Assouad dimension, we see that the cardinality $N_{n}$ of all dyadic cubes
of generation $n$ that intersect with the set $K$ is bounded above by $%
C_{\beta }\left( \sigma ^{n}\right) ^{-\beta }$ for some constant $C_{\beta
}>0$. Thus,
\begin{equation*}
\dim _{M}\mathcal{F}\leq \lim \frac{\log C_{\beta }\left( \sigma ^{n}\right)
^{-\beta }}{\log \frac{1}{\sigma ^{n}}}=\beta .
\end{equation*}%
This shows that $\dim _{M}\mathcal{F}_{K}\leq \dim _{A}\left( K\right) $.

\begin{proposition}
Suppose $\left( X,d\right) $ is a complete geodesic doubling metric space.
If $\mu $ is a probability measure concentrated on a bounded subset $K$ of $%
X $, then
\begin{equation*}
\dim _{M}\left( \mu \right) \leq \dim _{A}\left( K\right) .
\end{equation*}%
If in addition, $\mu $ is Ahlfors regular, then
\begin{equation*}
\dim _{U}\left( \mu \right) \leq \dim _{A}\left( K\right) .
\end{equation*}

\end{proposition}

\begin{proof}
Since $\mu $ is concentrated on $K$, we have $\mu $ is concentrated on the
associated nested collection $\mathcal{F}_{K}$. Thus,%
\begin{equation*}
\dim _{M}\left( \mu \right) \leq \dim _{M}\mathcal{F}_{K}\leq \dim
_{A}\left( K\right) .
\end{equation*}%
When $\mu $ is Ahlfors regular, $\mu $ is evenly concentrated on $\mathcal{F}%
_{K}$, thus $\dim _{U}\left( \mu \right) \leq \dim _{A}\left( K\right) $.
\end{proof}

In particular, we have

\begin{theorem}
Suppose $X$ is a complete geodesic doubling metric space with Assouad
dimension $m$, and $\mu $ is any probability measure on $X$ with a compact
support. Let $1-\frac{1}{m}<\alpha <1$. Then,

\begin{enumerate}
\item $\mu \in \mathcal{D}_{\alpha }\left( X\right) $ if $\alpha >0$. In
particular, if in addition $X$ is compact, then $\mathcal{D}_{\alpha }\left(
X\right) =\mathcal{P}_{\alpha }\left( X\right) =\mathcal{P}\left( X\right) $.

\item $\mu \in \mathcal{D}_{\alpha }\left( X\right) $ if $\mu $ is Ahlfors
regular.
\end{enumerate}
\end{theorem}

\begin{proof}
Let $K$ be the support of $\mu $. Then, $\dim _{M}\left( \mu \right) \leq
\dim _{A}\left( K\right) \leq \dim _{A}\left( X\right) =m$. By theorem \ref%
{main theorem}, $\dim _{T}\left( \mu \right) \leq \max \left\{ 1,\dim
_{M}\left( \mu \right) \right\} \leq \max \left\{ 1,m\right\} $. Therefore,
for any $\max \left\{ 1-\frac{1}{m},0\right\} <\alpha <1$, we have $\frac{1}{%
1-\alpha }>\max \left\{ 1,m\right\} \geq \dim _{T}\left( \mu \right) $, and
thus $\mu \in \mathcal{D}_{\alpha }\left( X\right) $. When $\mu $ is Ahlfors
regular on $X$, we have $\dim _{T}\left( \mu \right) \leq \dim _{U}\left(
\mu \right) \leq \dim _{A}\left( K\right) \leq m$. Thus, if $\frac{1}{%
1-\alpha }>m$, then $\mu \in \mathcal{D}_{\alpha }\left( X\right) $.
\end{proof}

Thus, by proposition \ref{p_alpha_geodesic}, we have

\begin{corollary}
Suppose $X$ is a compact geodesic doubling metric space with Assouad
dimension $m$. Then, the space $\left( \mathcal{P}\left( X\right) ,d_{\alpha
}\right) $ of probability measures on $X$ is a complete geodesic metric
space whenever $\max \left\{ 1-\frac{1}{m},0\right\} <\alpha <1$. In other
words, there exists an $\alpha -$optimal transport path between any two
probability measures on $X$.
\end{corollary}


\begin{thebibliography}{99}
\bibitem{Ambrosio} L. Ambrosio. Lecture notes on optimal transport problems.
Mathematical aspects of evolving interfaces (Funchal, 2000), 1--52, Lecture
Notes in Math., 1812, Springer, Berlin, 2003.

\bibitem{buttazzo} A. Brancolini, G. Buttazzo, F. Santambrogio, Path
functions over Wasserstein spaces. J. Eur. Math. Soc. Vol. 8, No.3
(2006),415--434.

\bibitem{metricgeometry} D. Burago, Y. Burago, S. Ivanov, A Course in Metric
Geometry, American Mathematical Society, 2001.


\bibitem{BCM} M. Bernot; V. Caselles; J. Morel, Traffic plans. Publ. Mat. 49
(2005), no. 2, 417--451.

\bibitem{book} M. Bernot; V. Caselles; J. Morel; Optimal Transportation
Networks: Models and Theory. Series: Lecture Notes in Mathematics , Vol.
1955 , (2009).

\bibitem{Brenier} Y. Brenier. D\'{e}composition polaire et r\'{e}arrangement
monotone des champs de vecteurs. C. R. Acad. Sci. Paris S\'{e}r. I Math. 305
(1987), no. 19, 805--808.

\bibitem{caffarelli} L.A. Caffarelli; M. Feldman; R. J. McCann. Constructing
optimal maps for Monge's transport problem as a limit of strictly convex
costs. J. Amer. Math. Soc. 15 (2002), no. 1, 1--26

\bibitem{cheeger} J. Cheeger, Differentiability of Lipschitz functions on metric
measure spaces, Geom. Funct, Anal. 9 (1999),  pp.428-517.

\bibitem{christ6} M. Christ: Lectures on singular integral operators. -
Conference Board of the Mathematical Sciences, Regional Conference Series in
Mathematics 77, 1990.

\bibitem{christ7} M. Christ. A T(b) Theorem with remarks on analytic
capacity and the Cauchy integral. Colloq. Math. LX/LXI:2, 1990, 601--628.

\bibitem{coifman} R. Coifman, and G. Weiss: Analyse Harmonique
Non-Commutative sur Certains Espaces Homogenes. - Lectures Notes in Math.
242, Springer--Verlag, 1971.

\bibitem{DH} Thierry De Pauw and Robert Hardt. Size minimization and
approximating problems, Calc. Var. Partial Differential Equations 17 (2003),
405-442.

\bibitem{Solimini} G. Devillanova and S. Solimini. On the dimension of an
irrigable measure. Rend. Semin. Mat. Univ. Padova 117 (2007), 1--49.

\bibitem{evan2} Evans, Lawrence C.; Gangbo, Wilfrid. Differential equations
methods for the Monge-Kantorovich mass transfer problem. Mem. Amer. Math.
Soc. 137 (1999), no. 653.

\bibitem{mccann} Gangbo, Wilfrid; McCann, Robert J. The geometry of optimal
transportation. Acta Math. 177 (1996), no. 2, 113--161.

\bibitem{gilbert} E.N. Gilbert, Minimum cost communication networks, Bell
System Tech. J. 46, (1967), pp. 2209-2227.

\bibitem{heinonen} J. Heinonen, Lectures on Analysis on Metric Spaces,
Universitext, Springer, 2001.

\bibitem{kantorovich} L. Kantorovich. On the translocation of masses. C.R.
(Doklady) Acad. Sci. URSS (N.S.), 37:199-201, 1942.

\bibitem{msm} F. Maddalena, S. Solimini and J.M. Morel. A variational model
of irrigation patterns, Interfaces and Free Boundaries, Volume 5, Issue 4,
(2003), pp. 391-416.

\bibitem{monge} G. Monge. M\'{e}moire sur la th\'{e}orie des d\'{e}blais et
de remblais, Histoire de l'Acad\'{e}mie Royale des Sciences de Paris,
666-704 (1781).

\bibitem{paolini} E. Paolini and E. Stepanov. Optimal transportation
networks as flat chains. Interfaces and Free Boundaries, 8 (2006), 393-436.

\bibitem{villani} Villani, C\'{e}dric. Topics in mass transportation. AMS
Graduate Studies in Math. 58 (2003)

\bibitem{white} B. White. Rectifiability of flat chains. Annals of
Mathematics 150 (1999), no. 1, 165-184.

\bibitem{xia1} Q. Xia, Optimal paths related to transport problems.
Communications in Contemporary Mathematics. Vol. 5, No. 2 (2003) 251-279.

\bibitem{xia2} Q. Xia. Interior regularity of optimal transport paths.
Calculus of Variations and Partial Differential Equations. 20 (2004), no. 3,
283--299.

\bibitem{xia3} Q. Xia. Boundary regularity of optimal transport paths.
Preprint.

\bibitem{xia4} Q. Xia. The formation of tree leaf. ESAIM Control Optim.
Calc. Var. 13 (2007), no. 2, 359--377.

\bibitem{xia5} Q. Xia. The geodesic problem in quasimetric spaces. Journal
of Geometric Analysis: Volume 19, Issue2 (2009), 452--479.

\bibitem{xia6} Q. Xia and A. Vershynina. On the transport dimension of
measures. arXiv:0905.3837

\end{thebibliography}
\end{document}